\documentclass[leqno,11pt, a4]{amsart}
\usepackage[utf8]{inputenc}
\usepackage{amsthm}
\usepackage{amsmath}
\usepackage{enumitem}
\usepackage{amsfonts}
\usepackage{amssymb}
\usepackage{graphicx}
\usepackage{pdfsync}
\usepackage{tikz}
\usetikzlibrary{shapes, arrows.meta, positioning}
\newtheorem{theorem}{Theorem}[section]
\newtheorem{definition}{Definition} [section]

\newtheorem{exmp}{Example}[section]

\newcounter{yuppo}
\renewcommand{\setminus}{-}

\renewcommand{\phi}{\varphi}

\newcommand{\noparty}[1]{}
\title{The momentum polytope of actions on contact manifolds via toric varieties}
\author{Amna Shaddad}
\address{Instituto de Ciencias Matem\'{a}ticas (ICMAT) \\
	CSIC (Madrid)}
\email{amna\_sh@msn.com; amna.shaddad@icmat.es}

\begin{document}
	\maketitle
	\begin{abstract}
		\noindent

		Due to a previous result which states that contact varieties are isomorphic to certain varieties, the momentum polytopes of contact manifolds are convex.
	\end{abstract}
	
	\section*{Institution and mail address}
	Instituto de Ciencias Matem\'{a}ticas (ICMAT)\\
	{\fontfamily{cmtt}\selectfont Office 417, Calle Nicol\'{a}s Cabrera 13-15, 28049, Cantoblanco, Madrid, Spain\\
		amna\_sh@msn.com; amna.shaddad@icmat.es}
	
	
	
	
	
	\newpage 
	
	\section{Introduction}
	
	The introduction of symplectic geometry comes from the investigation of light rays through a medium, this was extended to every mechanical system. One thing remained constant - each ray would be described according to two pieces of information - the position the ray is incident with the medium and the angle at which the ray hits the medium. According to Hamilton's paper `Caustics', the optical system used to study the geometry of rays of light, utilises four variables to locally specify the rays that enter a system from the left and exit from the right (the light rays to the left are straight line segments); two of the variables specify the point of intersection of the line with the plane perpendicular to the optical axis and two additional angular variables giving the angle of the line to this plane. Relating the incoming line segments on the left to the outgoing line segments on the right, is a transformation from the incoming coordinates to the outgoing ones, called symplectic diffeomorphisms.
	Later Hamilton discovered that this same method applies without modification to mechanics. Replacing the optical axis by the time axis, light rays  for trajectories of the system and the four incoming and four outgoing variables by the $2n$ incoming and outgoing variables of the phase space of the mechanical system.
	
	And therefore symplectic geometry was birthed and the associated momentum map and its image accordingly.
	
	Contact systems exist on $2n-1$ or odd dimensional manifolds. Contact geometry has been used to describe to many physical phenomena and is related to a lot of other mathematical structures. Sophus Lie was the first to introduce them through his work on partial differential equations. Then Gibbs reintroduced them through thermodynamics, Huygen's later through geometric optics and Hamiltonian dynamics. More recently contact structures have been found to describe heat transport through a medium. But this was only after previous work relating contact structures to Riemannian geometry, low dimensional topology and subelleptic operators.
	
	Focussing on cooriented contact structures i.e. hyperplane fields $\xi$ on oriented smooth $2n-1$-dimensional manifolds which are given by the kernels of $\alpha\in\Omega^1(M)$ so that $\alpha\wedge d\alpha^{n-1}$ is a positive volume form on said manifold. We establish the momentum polytope for such structures.
	
	We first introduce some notation concerning toric varieties in particular. We also talk about remedying singular points of toric varieties. Later we introduce how these apply to orbits, line bundles and topology. What is especially significant here is the introduction of $\chi(u)$, the element of the group algebra of semigroup $S$ defined by the cone in the lattice $N$, $\mathbb{C}[S]$ corresponding to $u$ in $S$. For the momentum map is defined according to $\chi(u)$ as a retraction. This is shown to establish a momentum polytope for actions on contact manifolds by way of example in the final section. Such a momentum map works independently of dimensional and some structural restrictions on the manifold acted on.
	
	\subsection{Main Result}
	
	The result of this paper is
	
	\begin{theorem}
		For $X$ a smooth projective toric variety of dimension $2n+1$ $(n\geq 1)$, defined over the field $\mathbb{C}$ of complex numbers and endowed with a contact structure. Then either $X$ is isomorphic to the complex projective space $\mathbb{P}^{2n+1}$ or $X$ is isomorphic to the variety $\mathbb{P}_{\mathbb{P}^1\times\cdots\times\mathbb{P}^1}(\mathcal{T}_{\mathbb{P}^1\times\cdots\times\mathbb{P}^1})$. The momentum map is defined $$\mu(x)=\frac{1}{\sum|\chi^u(x)|}\sum\limits_{u\in P\cap M}|\chi^u(x)|u$$ where $\chi^u$ for $u\in P\cap M$ are the sections of $\varphi:X\rightarrow\mathbb{P}^{2n}$. This map is convex.
	\end{theorem}
	
	\subsection{Literature Review}

	A hyperplane field $\xi$ on a manifold $M$ is a codimension one sub-bundle of the tangent bundle $TM$. Hyperplane fields can always locally be described as the kernel of a 1-form. This means that for every point in $M$, there exists a neighbourhood $U$ and a 1-form $\alpha$ defined on $U$ whose kernel of the linear map $\alpha_x:T_xM\rightarrow\mathbb{R}$ is $\xi_x$ for all $x\in U$. $\alpha$ is the local defining form for $\xi$ and a contact structure on a $(2n+1)$-dimensional manifold $M$ is a 'maximally non-integrable hyperplane field' $\xi$. If for any locally defining 1-form $\alpha$, $\alpha\wedge d\alpha\neq0$ then the hyperplane field $\xi$ is said to be maximally non-integrable. Essentially what this says is that the form is pointwise never equal to zero. This non-integrability of $\xi$ geometrically means that no hypersurface in $M$ can be tangent to $\xi$ along an open subset of the hypersurface. In other words, the hyperplane 'twists too much' to be tangent to the hypersurface. The pair $(M,\xi)$ is named a contact manifold and $\alpha$, the defining form for $\xi$ is the contact form for $\xi$.
	
	Let's look at the relations that exist between contact geometry and symplectic geometry. For symplectic manifold $(S,\omega)$, $L_v\omega=\omega$ for vector field $v$, where $L_v\omega$ is the Lie derivative of $\omega$ in the direction of $v$, and this is called a symplectic dilation. In $(S,\omega)$, compact hypersurface $M$ is said to have contact type if there exists a symplectic dilation $v$ in a neighbourhood of $M$ that is transverse to $M$. The characteristic line field $LM$ in the tangent bundle of given hypersurface $M$ in $(S,\omega)$, is the symplectic complement of $TM$ in $TS$. $M$ is therefore coisotropic since it is codimension one, and therefore the symplectic complement is contained in $TM$ and is of dimension one.
	
	For $M$ a compact hypersurface in a symplectic manifold $(S,\omega)$, with inclusion map $i:M\rightarrow M$, then $M$ has contact type if and only if a 1-form $\alpha$ exists on $M$ so that $d\alpha=i^*\omega$ and $\alpha$ is never zero on the characteristic line field.
	
	So if $M$ is a hypersurface of contact type, then $\alpha$ is found by contracting the symplectic dilation $v$ into the symplectic form i.e. $\alpha=\iota_v\omega$. It is easy to prove that the 1-form $\alpha$ is a contact form on $M$ and therefore a hypersurface of contact type in a symplectic manifold has a co-oriented contact structure.
	
	So provided a co-orientable contact manifold $(M,\xi)$, a symplectisation $\text{Symp}(M,\xi)=(S,\omega)$ can be formulated. Meaning that for manifold $S=M\times(0,\infty)$, given $\alpha$ a global contact form for $\xi$, then $\omega=d(t\alpha)$ where $t$ is the coordinate on $\mathbb{R}$. one may also symplectise as $(M\times\mathbb{R},d(e^t\alpha))$. The symplectisation is independent of the choice of $\alpha$.
	
	So for a co-oriented contact manifold $(M\xi)$, a symplectic manifold $\text{Symp}(M,\xi)$ exists in which $M$ sits as a hypersurface of contact type. and in fact, contact form $\alpha$ for $\xi$ provides an embedding of $M$ into this symplectic manifold that recognises $M$ as a hypersurface of contact type. Given $M$, a compact hypersurface of contact type in a $(S,\omega)$, with symplectic dilation $v$. Then there if a neighbourhood of $M$ in $S$ symplectomorphic to a neighbourhood of $M\times\{1\}$ in $\text{Symp}(M,\xi)$ where the symplectisation is simply $M\times(0,\infty)$ using $\alpha=\iota_v\omega|_M$ and $\text{ker}\alpha=\xi$.
	
	Let $(M,\xi)$ be a contact manifold. Associated to $\alpha$ to $\xi$ is the Reeb vector field $v_\alpha$ which is a unique vector field that satisfies $\iota_{v_\alpha}\alpha=1$ and $\iota_{v_\alpha}d\alpha=0$. $v_\alpha$ is certainly transverse to the contact hyperplanes and $\xi$ is preserved by the flow of $v_\alpha$. These two conditions characterise Reeb vector fields. Meaning that a vector field $v$ is the Reeb vector field for some contact form for $\xi$ if and only if it is transverse to $\xi$ and its flow preserves $\xi$. 
	
	Yi Lin and Reyer Sjamaar's result in \cite{LinSjamaar:2017dg} applies, to Elisa Prato's quasifolds and to Hamiltonians on contact manifolds and cosymplectic manifolds. In their paper they define a presymplectic manifold $(\mathcal{X},\omega)$ as a paracompact $C^\infty$-manifold that has a closed 2-form of constant rank, and a Hamiltonian action on a presymplectic manifold consists of a smooth action of a Lie group $G$ on $\mathcal{X}$ and a smooth moment map $\Phi:\mathcal{X}\rightarrow\mathfrak{g}^*$. Theorem 3.4.6 in their paper says that by assuming that the $G$ action on $\mathcal{X}$ is clean. By clean they mean of Reeb type. According to them cleanness means that there should exist an ideal of the Lie algebra of the groups, called the 'null ideal', which at every point of the manifold spans the tangent space of the intersection of the group orbit with the leaf of $\mathcal{F}$.
	S, if the action of clean then for every $\epsilon\in\mathfrak{g}$, $\Phi^\epsilon$ is the corresponding Morse-Bott function. And the positive and negative bundles of the critical set $\mathcal{X}^{[\epsilon]}$ with respect to the $G$-invariant compatible Riemannian metric on $\mathcal{X}$ - the symplectic subbundles of $T\mathcal{X}$ are orthogonal to the subbundle. Section 4.5 of that paper then introduces contact manifolds as one of the presymplectic manifolds to which their results apply on which the action of $G$ leaves $\alpha$ invariant. The Hamiltonian action is defined as $\Phi^\epsilon=\iota(\epsilon_\mathcal{X})\alpha$. The action is of Reeb type if there exists $\epsilon\in\mathfrak{g}$ with $\epsilon_\mathcal{X}=v_\alpha$, where $v_\alpha$ is the Reeb vector field introduced before. The Reeb type actions are leafwise transitive. The Reeb vector field spans the null foliation $\mathcal{F}$ of $\omega$.
	
	Lin and Sjamaar's presymplectic convexity theorem applied to contact structures, aptly named the 'Contact convexity theorem' part I - theorem 4.5.1 is that if the action of $G$ on $\mathcal{X}$ is clean, i.e. of Reeb type, then $\Phi(\mathcal{X})\cap C$ is a convex polytope denoted $\triangle(\mathcal{X})$. There exist many compact contact Hamiltonian $G$-manifolds $X$ where $\triangle(\mathcal{X})$ is not convex which implies that the action is not clean. 
	
	They propose several methods for examples, for a manifold $(\mathcal{X},\alpha)$, where $\triangle(\mathcal{X})$ is convex, if $\alpha$ is replaced with a conformally equivalent contact form $e^f\alpha$ where $f$ is a $G$-invariant smooth function i.e. $\Phi\times e^f$ then this destroys the convexity of $\triangle(\mathcal{X})$.
	And another a method for a failure of convexity: Zhenqi He gave the first example of a presymplectic Hamiltonian torus action with nonconvex momentum map in his PhD thesis titled 'Odd dimensional symplectic manifolds' completed at MIT. Such actions are not typically clean. Introducing a hypersurface $Y\subset\mathbb{R}^d$ that is transverse to the orthant faces of $\mathbb{R}^d$ and the radial vector field on $\mathbb{R}^d$. For a quadratic momentum map $\phi$ this implies that radial vector fields on $\mathbb{C}^d$ map to two times the radial vector fields on $\mathbb{R}^d$. For $\phi^{-1}Y$, it is transverse to the Liouville radial vector field on $\mathbb{C}^d$ which implies that $\phi^{-1}Y$ is of contact type. Only if $Y$ is an affine hyperplane does this imply that $\phi^{-1}Y$ is a contact ellipsoid and $\Phi(\phi^{-1}Y)$ is a simplex. Therefore $Y\cap\mathbb{R}^d_{\geq0}$ is not convex.
	Now for Lin and Sjamaar their results regard the symplectic leaf space $\mathcal{X}/\mathcal{F}$ of $\mathcal{X}$. For Chiang and Karshon's improvement \cite{ChiangKarshonLerman:2010dg} of Lerman's theorem in \cite{Lerman:2002dg}, their results regard the symplectisation as described before but where $\omega_S=-d(t\alpha)=t\omega-dt\wedge\alpha$ for $t\in(0,\infty)$. For trivial action of $G$, the Hamiltonian $G$-action on $S$ with $\Phi_S(x,t)=t\Phi(x)$ and the $\text{Im}(\Phi_S)$ is the conical set $\Phi_S(S)=\bigcup_{t>0}t\Phi(M)$. The subset is defined as $\text{Im}(\Phi_S)\cap C$ plus the origin: $\triangle(S)=\{0\}\cup(\Phi_S(S)\cap C)$ of $C$ meaning that $\triangle(S)=\bigcup_{t\geq0}t\triangle(M)$ i.e. it is the union of all dilations of $\triangle(S)$. \emph{Then} if $G$ is a torus of dimension greater than or equal to 2, this implies that $\triangle(S)$ is convex but $\triangle(M)$ is not. Another difference between \cite{LinSjamaar:2017dg} and \cite{ChiangKarshonLerman:2010dg} is that $\text{Im}(\Phi(M))$ is highly dependent on the choice of contact form $\alpha$. But the symplectic cone $(S,\omega_S)$ is intrinsic and invariant of the contact hyperplane bundle $\text{ker}(\alpha)$ which implies that $\triangle(S)$ depends on the conformal class of $\alpha$.
	
	However both these papers are not wholly disconnected from each other. The connection between theorem 4.5.1 of \cite{LinSjamaar:2017dg} and \cite{ChiangKarshonLerman:2010dg} is that the Reeb vector field $v_\alpha$ satisfies $\iota_{v_\alpha}\omega_S=-t\iota_{v_\alpha}d\alpha-\iota_{v_\alpha}(dt\wedge\alpha)=0+dt\wedge\iota_{v_\alpha}\alpha=dt$ i.e. it's the Hamiltonian vector field of $t$ on cone $S$. As $v_\alpha$ is $G$-invariant this implies that $S$ is a Hamiltonian $\hat{G}$-manifold for $\hat{G}=G\times\mathbb{R}$ and $\hat{\Phi}:S\rightarrow\mathfrak{g}^*\times\mathbb{R}$ where $\hat{\Phi}(x,t)=(t\Phi(x),t)$. If we have $\hat{\triangle}(S)=\{0\}\cup(\hat{\Phi}(S)\cap(C\times\mathbb{R}))\subset C\times\mathbb{R}$ and $\hat{\triangle}(S)=\bigcup_{t\geq0}(t\triangle(M)\times\{t\})$ and $\triangle(M)=\hat{\triangle}(S)\cap\{t=1\}$ as it should be! As the symplectic quotient of $S$ at the level 1 with respect to the $\mathbb{R}$-action is leaf space $X/\mathcal{F}$. Therefore $\triangle(S)$ is the restriction of $\hat{G}$-action to $G$ i.e. projecting $\hat{\triangle}(S)$ along the $\mathbb{R}$-axis. Therefore leading to a nonabelian extension of the Lerman-Chiang-Karshon theorem (theorem 4.5.3 of \cite{LinSjamaar:2017dg}) which states that $\hat{\triangle}(S)$ and $\triangle(S)$ are convex polyhedral cones for a clean $G$-action on $M$. This proof does not work for cone $\triangle(S)$ rational nor does it prove convexity if the action is not clean.
	
	Looking more closely at \cite{Lerman:2002dg} where the symplectic cone is the symplectisation of the contact manifold, $$\langle\Phi(q,p),\text{A}\rangle=\langle p,\text{A$_m$}(q)\rangle$$ where $\text{Am}$ is the vector field induced on $S$ by $\text{A}\in\mathfrak{g}$. The momentum cone $C(\Upsilon)=\Upsilon(\xi^\circ_+)\cup\{0\}$ with contact distribution $\xi$ and $\xi^\circ\subset T^*S$. The $G$-action preserves $\xi$ coorientation which implies that the two components of $\xi\setminus0$ is preserved. $\xi^\circ_+$ is one component. $\Upsilon=\Phi|_{\xi^\circ_+}$ where $\Upsilon:\xi^\circ_+\rightarrow\mathfrak{g}^*$. $\Upsilon_\alpha$ is the momentum  map defined by the contact form $\alpha$ and $\Upsilon$ is the momentum map defined by the contact distribution $\xi$. The proof of lemma 1.7 in \cite{Lerman:2002dg} uses the fact that $\text{dim}S+1=2\text{dim}G$ which implies that $\Upsilon_\alpha(x)\neq0$ for all $x\in S$. If we suppose not then for some point $x\in S$ ($\Upsilon_\alpha(x)=0$) the orbit $G\cdot x$ is tangent to the contact distribution. See remark 1.9 in that paper. It is a condition on a contact distribution and not on a particular choice of a contact form representing the distribution. Chiang and Karshon answer the question, is the transversality condition necessary?:
	``Eugene Lerman gave analogous theorem in equivariant contact geometry when torus orbits are transverse to the contact distribution, and asked whether transversality...in this paper we answer Lerman's question and give the 'convexity package'".
	The solution of which is in remark 1.2 in \cite{ChiangKarshonLerman:2010dg}. This paper is a step by step of several proofs for turning one thing into another - namely providing as they said a convexity package for contact structures.
	
	Boyer and Galicki in \cite{BoyerGalicki:1999dg} mainly prove Delzant's theorem: Every compact toric contact manifold whose Reeb vector field corresponds to an element of the Lie algebra of the torus (Reeb) can be procured by contact reduction from an odd dimensional sphere. For toric contact geometry, if a torus action is regular this means that the image is a sphere. And if the action is singular then the image is a closed convex poytope. According to Banyanga and Molino in CITE the polytope determines a contact toric structure up to isomorphism. But not that every such manifold can be procured from reduction (this may not be true). So if one makes the added assumption that the Reeb vector field corresponds to the element of the Lie algebra of the torus then it does become true. Their method, starts with the Atiyah-Guillemin-Sternberg theorem holding true for compact contact manifolds of Reeb type. Then if the contact 1-form is Pfaffian this shows that the convex polytope on certain hyperplane is the characteristic hyperplane in the dual of the torus. If we change the Pfaffian function, this changes the polytope by a scale. It is always possible to choose a Pfaffian so the polytope is rational. Now, according to Lerman and Tolman's theorem in CITE The geometry of their labelled polytopes are directly related to the geometry of the contact manifold with fixed Pfaffian with toris action and characteristic foliation. According to section 2 of that paper, if we fix the contact structure by an equivalence class this provides a Pfaffian structure. The proposition therein proposes that contact manifold $(S,\xi)$ is orientable if and only if $C(S)=S+\mathbb{R}^+$ is a symplectic cone. Fixed $\xi$ has a unique characteristic/Reeb vector field (as we know) but if $v_\alpha$ is nowhere vanishing then a foliation is induced that is defined $\mathcal{F}^{v_\alpha}$ of $S$. It is called the characteristic foliation ($v_\alpha$ and $\mathcal{F}^{v_\alpha}$ depend on a choice of $\xi$).
	
	In \cite{BoyerGalicki:1999dg}, $\mathfrak{C}(S,\mathcal{D})$ is the group of contact transformations (where $\mathcal{D}$ denotes the contact distribution in this case), a subset of $\text{Diff}(S)$ that leave $\mathcal{D}$ invariant. By fixing $\mathcal{D}=\text{ker}\xi$ which implies that $\mathfrak{C}$ is a subgroup $\phi:M\rightarrow M$ of diffeomorphisms with $\phi^*\xi=f\xi$ for $f\neq0$. What's wanted is a $\mathfrak{C}(M,\xi)$  for $\phi^*\xi=\xi$ with $\xi$ fixed. Which is called a strict contact transformation.
	
	For $\mathfrak{C}(S,\mathcal{D})$ and $\mathfrak{C}(S,\xi)$ the Lie algebras are $\mathfrak{c}(S,\mathcal{D})$ and $\mathfrak{c}(S,\xi)$ respectively. And the respective Lie algebras of the group of symplectomorphisms are $\mathfrak{s}(\mathfrak{c}(S,\omega))$ and $\mathfrak{s}_0(\mathfrak{c}(S,\omega))$. And on $C(S)$, $\mathfrak{G}(C(S),\omega)$ denote the group of symplectomorphisms of $(C(S),\omega)$, also $\mathfrak{G}_0(C(S),\omega)$ denote the subgroup of $\mathfrak{G}(C(S),\omega)$ that commutes with the homotheties meaning that it is the automorphism group of the symplectic Liouville structure. 
	
	They proposed that $\Phi(\mathcal{X})=\xi(\mathcal{X})$ where $\xi(\mathcal{X})$ the contact Hamiltonian function, and $\Phi\mathcal{X}$ the Lie algebra isomorphism between $\mathfrak{c}(S,\mathcal{D})$ of infinitesimal contact transformations. And $\mathcal{X}\mapsto\mathcal{X}_S\mapsto\xi(\mathcal{X}_S)$ implies that $\mathfrak{s}(C(S),\omega)\approx\mathfrak{c}(S,\xi)\approx C^\infty(S)^{v_\alpha}$ where $v_\alpha$ is the center of $\mathfrak{c}(S,\xi)$.
	
	The Atiyah-Guillemin-Sternberg theorem holds for compact contact manifolds of Reeb type. $\mathfrak{G}$ is the Lie group acting on symplectic cone $(C(S),\omega)$ leaves with invariant, commutes with homothety group. $\mathfrak{G}\subset\mathfrak{G}_0$ and $\tilde{\mu}:C(S)\rightarrow\mathfrak{g}^*$ where $\tilde{\mu}^\tau=\tilde{\xi}(\mathcal{X}_\tau)$. The homotheties are such that $r\mapsto e^tr$, $\omega\mapsto e^t\omega$, $\tilde{\xi}\mapsto e^t\tilde{\xi}$ and $\tilde{\mu}\mapsto e^t\tilde{\mu}$. Consider the $n+1$-dimensional torus, $\mathbb{T}^{n+1}$, $\{e_i\}^n_{i=0}$ is the standard basis for $\mathfrak{t}_{n+1}\simeq\mathbb{R}^{n+1}$ for $e_i$ have $\mathcal{X}^{e_i}$ denote $H_i$. $v_\alpha$ corresponds to $\zeta$, the characteristic vector, in $\mathfrak{t}_{n+1}$. So, the Reeb vector is almost periodic. So the torus action is of Reeb type. The dual basis of $\mathfrak{t}_{n+1}^*$ is $\{e^*_i\}^n_{i=0}$. $\langle \mu,\zeta\rangle=\sum^i_{i=0}\mu_ia_i=1$ the characteristic hyperplane of codimension 1 ($\zeta=\sum^n_{i=0}a_ie_i$, $\mu=\sum^n_{i=0}\mu_ie^*_i$).
	
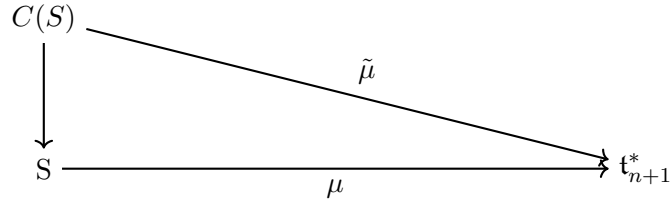
	\begin{center}
		\begin{tikzpicture}
			\node (A) at (-5,0) {$C(S)$};
			\node (B) at (-5,-2) {S};
			\node (C) at (3,-2) {$\mathfrak{t}^*_{n+1}$};

			\draw [thick, ->] (A) -- (B);
			\draw [thick, ->] (A) -- (C) node[midway,above right] {$\tilde{\mu}$};
			\draw [thick, ->] (B) -- (C) node[midway,below] {$\mu$};
		\end{tikzpicture}
	\end{center}
	
	where the inclusion $S$ into $C(S)$ as $S\times{1}$.
	
	They defined the facet as a codimension one face. A dimension $n$ simple convex polytope - $n$ facets meeting at each vertex. And a rational convex polytope as $\Delta=\cap_{i=1}^N\{\alpha\in\mathfrak{t}^*|\langle\alpha,y_i\rangle\leq\lambda_i\}$ where $l\subset\mathfrak{t}$ is the lattice of the circle subgroups of $\mathbb{T}$, where $N$ is the number of faces.
	
	By the above, polytope $\mu(S)\subset\mathfrak{t}^*$ is simple, of dim $\mathfrak{t}-1$, and rational if and only if characteristic vector $\zeta$ lies in lattice $l\subset\mathfrak{t}$ of circle subgroups of $\mathbb{T}$.
	
	Theorem 4.5 of \cite{BoyerGalicki:1999dg} states for $x\in S$:\\
	(I) the naming and lettering of facets and vectors. And the polytope is the convex hull of vertices.\\
	(II) $\mathfrak{l}_x\in\mathfrak{h}_x$ (Lie algebra of isotropy subgroup $\mathbb{H}_x$ of $\mathbb{T}$ at $x$, is the linear span of vectors $p_i \forall i$ such that $i^\text{th}$ open facet $f_i$ lies in $F(x)$ (set of open faces)). $\mathfrak{l}_x$ denotes lattice of circle subgroups of $\mathbb{H}_x$ and $\hat{\mathfrak{l}}_x$ denotes sublattice of $l$ generated by $\{m_ip_i\}_{f_i\in F(x)}$ generates $\mathcal{l}_x$ $\forall x\in M$ (here $m_i=1$ $\forall i$).
	
	Then they discuss varying $\xi'=f\xi$ with a corresponding lemma.
	
	$(n+1)$-torus $\mathbb{T}^{n+1}$ $\zeta$ in $\mathfrak{t}_{n+1}$ $\mu_i=\xi(H_i)$. The polytope for $\xi$ and the polytope for $\xi'$, the same number of facets and vertices but size is different (depends on Pfaffian). But in $R\mathbb{P}^{n+1}$ lines through origin in $\mathfrak{t}_{n+1}$ intersect lattice $\mathcal{l}$ of circle subgroups are dense. So, perturb Reeb vector field and contact form. Which means that the foliation is quasi-regular and that the polytope is rational.
	After defining the toric contact manifold of Reeb type $(S,\mathcal{D},\mathbb{T})$ they explained the proposition that $(S,\mathcal{D},\mathbb{T})$ implies that $\mathcal{D}$ has quasi-regular contact form i.e. rational polytope. 
	
	So, for a $(S,\xi)$ a compact toric contact manifold of Reeb types and fixed quasi-regular contact form $\eta$, then $(S,\xi)\simeq$ reduction by a torus of $S^{2n+1}$ sphere with $\xi_a$. The proof involves looking at the space of leaves $\mathcal{Z}$ of characteristic foliation $\mathcal{F}$. For $M$ toric this implies that $\mathcal{Z}$ is toric which implies $\mathbb{T}^n$ preserving $\omega$ on $\mathcal{Z}$. The integral class in $H^2_{orb}(\mathcal{Z},\mathbb{Z})$ ($S^1$ $V$-bundle $S\xrightarrow{\pi}\mathcal{Z})$. Chern class $d\xi=\pi^*\omega$ implies $H^2_{orb}(\mu^{-1}_{N-n}(n)/\mathbb{T}^{N-n},\mathbb{Z})$ ($S^1$ $V$-bundle $\pi:p\rightarrow \mu^{-1}_{N-n}(\lambda)\setminus\mathbb{T}^{N-n})$. With $\tilde{\pi}:P\rightarrow\mu^{-1}_{N-n}(\lambda)\setminus\mathbb{T}^{N-n}$ and $\phi:S\rightarrow P$ and $\tilde{\phi}\mathcal{Z}\rightarrow\mu^{-1}_{N-n}(\lambda)\setminus\mathbb{T}^{N-n}$ so $\phi^*d\hat{\xi}=d\xi$ means we can choose $\phi^*\hat{\xi}=\xi$. $\mathbb{T}^{N-n}=S^1_\zeta\times\mathbb{T}^{N-n-1}$ $P=\mu^{-1}_{N-n}(\lambda)/\mathbb{T}^{N-n-1}$ so $(S,\xi)$ becomes $(\mu^{-1}_{N-n}(\lambda)\setminus\mathbb{T}^{N-n-1},\hat{\xi})$. $\sum_a=\{\sum_ia_i|z_i|^2\}\simeq S^{2n+1}$. Defining $p:\mu^{-1}_{N-n}(\lambda)\rightarrow S$, so $p^*\xi=\iota^*\xi_a$. So $(S,\xi)$ is obtained from $(\sum_a,\xi_a)$ by contact reduction.
	
	Their main theorem is that for a compact toric contact manifold of Reeb type with a fixed quasi-regular contact form, the manifold is isomorphic to the reduction by a torus of a $2n-1$ dimensional sphere with its standard contact structure and with fixed 1-form (of particular Reeb type). Where every such manifold admits a compatible Sasakian structure. As every symplectic toric orbifold is in fact K\"{a}hler.
	
	Now not every contact vector field is the Reeb field of some contact form. Indeed, $v_\alpha$ is the Reeb field of some contact form which defines $\xi$ if and only if $v_\alpha$ is transverse to $\xi$, i.e. $\alpha(v_\alpha)\neq0$ for any defining form $\alpha$. 
	
	According to \cite{ChiangKarshonLerman:2010dg} a convex cone is defined as $C(\Psi):=\{0\}\cup\Psi(S\times\mathbb{R}_{>0})$ $S$ compact connected cooriented contact manifold equipped with an effective action of a torus of dimension $k>2$ and let $\Psi:S\times\mathbb{R}_{>0}\rightarrow\mathbb{R}^k$ is the momentum map on the symplectisation. Their symplectisation of $(S,\alpha)$ is the symplectic manifold $(S\times\mathbb{R}_{>0},d(t\alpha))$ where $t$ is the coordinate on $\mathbb{R}_{>0}$ and $\alpha$ is the contact one form on $S$ and its pullback to $S\times\mathbb{R}_{>0}$. The nondegeneracy of $d(t\alpha)$ follows from the property $\alpha\wedge(d\alpha)^n\neq0$. The positive connected component of the annihilator of $\xi$ on cotangent bundle $T^*M$ is (similar to before) $\xi^0_+:=\{(x,\beta)|x\in S,\ \beta\in T^*_xM,\ |\beta(\xi|_x)=0, \beta\ \text{induces coorientation of}\ \xi|_x\}$ and map $S\times\mathbb{R}_{>0}\rightarrow\xi^0_+$ sends $(x,t)$ to $t\alpha_x$. 
	
	In general the manifold $(Q\times\mathbb{R}_{>0},d(e^\theta\alpha))$ is called a symplectisation of $S$, where $\theta$ denotes the coordinate on $\mathbb{R}$. Every (transversally orientable) contact manifold $(Q,\xi)$ can be embedded as a hypersurface in an exact symplectic manifold. Choose any contact form $\alpha$. Then $M=Q\times\mathbb{R}$ is a symplectic manifold with symplectic form $\omega=e^\theta(d\alpha-\alpha\wedge d\theta)=d\lambda$, where $\lambda=e^\theta\alpha$.
	
	Going back to \cite{ChiangKarshonLerman:2010dg} they introduce momentum map $\Psi_\alpha: S\rightarrow\mathfrak{t}^*$ whose $\mathfrak{x}$ component $\Psi_\alpha^\mathfrak{x}:S\rightarrow\mathbb{R}$ for $\mathfrak{x}\in\mathfrak{t}$ is: $\Psi_\alpha^\mathfrak{x}(s)=\alpha(\mathfrak{x}_S(s))$ $\forall s\in S$ satisfying $d\Psi^\mathfrak{x}_\alpha=-\iota(\mathfrak{x}_S)d\alpha$ on $S$. Using $\alpha$ to identify $\xi^0_+$ with $S\times\mathbb{R}_{>0}$, induced torus action on $S\times\mathbb{R}_{>0}$ is the given action on $S$ component, which is trivial on the $\mathbb{R}_{>0}$ component. $\Psi:S\times\mathbb{R}_{>0}\rightarrow\mathfrak{t}^*$ with $(s,t)\mapsto t\Psi_\alpha(s)$ which is contact momentum map corresponding to $\alpha$. So the momentum cone is $C(\Psi)=\{0\}\cup\Psi(S\times\mathbb{R}_{>0})=\mathbb{R}_{\geq0}\cdot\Psi_\alpha(S)$.
	
	\cite{LinSjamaar:2017dg} state that the cone $C(\mathcal{X})$ is obtained by restricting the $\hat{G}$-action (defined above) to $G$, i.e. by projecting $\hat{C}(S)$ along the $\mathbb{R}$ axis. Where for $C(S)$ to be a convex polytope, $\hat{C}(S)$ is the convex polyhedral cone and its projection $C(S)$ onto $\mathfrak{t}^*$is a convex polyhedral cone. $C(\mathcal{X})=\Phi(\mathcal{X})\cap$positive Weyl chamber. $\hat{C}(S)=\bigcup_{t\geq0}(tC(\mathcal{X})\times\{t\})$ and $C(S)=\bigcup_{t\geq0}tC(\mathcal{X})$ and $\Phi_S(x,t)=t\Phi(x)$.
	
	Now $C(\mathcal{X})$ is rational if and only if the null subgroup $N(\mathcal{X})$ of $G$ is closed. $\langle \xi,\cdot\rangle\geq a$ is the half space or character lattice. According to \cite{Lerman:2002dg} a cone or polytope is rational if annihilators of codimension one faces are spanned by integral lattice i.e. vectors if and only if the limit of the kernel of exp$:\mathfrak{g}\rightarrow G$.
	
	Section 2.7 of \cite{LinSjamaar:2017dg} says that irrationality ker$(\omega)=0$ foliation is trivial but weakly rational i.e. normal vectors are contained in quasi-lattice $\mathfrak{X}_*(\mathbb{T})/\mathfrak{X}_*(\mathbb{T}\cap\mathfrak{n}(\mathcal{X}))$ in quotient space $\mathfrak{t}/\mathfrak{n}(\mathcal{X})$ where $\bar{x}=\mathcal{F}(x)$ and $G_{\bar{x}}$ is the stabiliser of $\bar{x}$ with corresponding Lie algebra $\mathfrak{g}_{\bar{x}}$ and $\mathfrak{n}(U)=\bigcap_{x\in G\cdot U}\mathfrak{g}_{\bar{x}}$ 
	for all open $U\subseteq\mathcal{X}$. $\mathfrak{n}(\mathcal{X})\cap\mathfrak{t}$ of $\mathfrak{t}$ is rational so subspace $\mathfrak{n}^*(\mathcal{X})^\circ\cap\mathfrak{t}^*$ is rational. By this, $C(\mathcal{X})$ is rational and $C(\mathcal{X})=\bigcap_{x\in\mathcal{X}}C_x\cap\mathfrak{n}(\mathcal{X})^\circ$. Remark 1.3 of Lerman's paper proves rationality - it is the main result of his paper, \cite{Lerman:2002dg}.
	And Banyango and Molina in \cite{BanyangaMolino:2011dg} prove rationality in Lemma 2.2 and Remark 2.3 of their paper, stating that for $W\subset S$ is a rational polyhedral if there exist vectors $v_1,\ldots,v_k$ in integral lattice $\mathbb{Z}_G=\text{ker}\{exp:\mathfrak{g}\rightarrow G\}$ such that $W=\{f\in S|\langle f,v_i\rangle\geq0,\ 1\leq i\leq k\}$.

	There are other papers that also tussle with the image of the momentum map of an action on a contact structure. More recently T. Ratiu and N. T. Zhung write in their preprint titled "Presymplectic convexity and (ir)rational polytope" that: "Battaglia and Prato and Katzarkov-Lupercio-Meerseman-Verjovsky have worked on irrational analogues of symplectic toric manifolds. However we wanted to have a simpler understanding of the geometric structure and so developed our own approach, which uses presymplectic realisations".
	
	

	\section{Background}

	\subsection{Toric Varieties}
	
	\subsubsection{Introduction}
	
	Toric varieties are geometric objects defined by combinatorial information. 
	
	A lattice $N\simeq \mathbb{Z}^n$, and real vector space $N_\mathbb{R}=N\otimes_\mathbb{Z}\mathbb{R}$. For a strongly convex rational polyhedral cone $\sigma$ in $N_\mathbb{R}$, rationality means that it is generated by a finite number of vectors and a convexity is strong if it doesn't contain a line through the origin. So $\sigma$ is defined as a cone in $N_\mathbb{R}$ that's generated by vectors in the lattice with apex at the origin. A fan $\Delta$ in $N$ is a collection of $\sigma$ where every face of a cone in $\Delta$ is also a cone in $\Delta$ and the intersection of two cones in $\Delta$ is a face of each.
	
	The dual cone denoted $\sigma^\smallsmile$ in $M_\mathbb{R}$, where $M=\text{Hom}(N,\mathbb{Z})$ is the dual lattice (with dual pairing $\langle,\rangle$), is the set of vectors that are nonnegative in $\sigma$. A finitely generated commutative semigroup is therefore established as $S_\sigma=\sigma^\smallsmile\cap M=\{u\in M:\langle u,v\rangle\geq0\ \text{for all}\ v\in\sigma\}$ where $\mathbb{C}[S_\sigma]$ is its corresponding group algebra - a finitely generated commutative $\mathbb{C}$-algebra. $U_\sigma=$Spec$(\mathbb{C}[S_\sigma])$ is the affine variety.
	
	For a face of $\sigma$, $\tau$ - the commutative semigroup $S_\tau$ contains $S_\sigma$ and $\mathbb{C}[S_\sigma]$ is a subalgebra of the associated group algebra $\mathbb{C}[S_\tau]$, and the affine variety $U_\tau$ is a principal open subset of $U_\sigma$, for $u\in S_\sigma$- so $\tau=\sigma\cap u^\perp$ then $U_\tau=\{x\in U_\sigma:u(x)\neq0\}$ giving the map $U_\tau\rightarrow U_\sigma$. These affine varieties form the algebraic variety $X(\Delta)$. Any product of projective and affine varieties is a toric variety. $T$ the torus of algebraic groups $\mathbb{C}^*\times...\times\mathbb{C}^*$ - a normal variety $X$ contains $T$ as a dense open subset and $T\times X\rightarrow X$. The simplest compact example is indeed the projective space $(\mathbb{C}^*)^n\subset\mathbb{C}^n\subset\mathbb{P}^n$.
	
	Now smaller cones mean smaller open sets therefore the geometry in $N$ is preferred to the geometry in $M$. 
	
	\subsubsection{Convex polyhedral cones}
	
	Next we will cover how to find generators of the semigroups. If we denote $N_\mathbb{R}$ as $V$ and its dual $V^*=M_\mathbb{R}$, then the convex polyhedral cone is the set $\sigma=\{r_1v_1+...+r_sv_s\in V:r_i\geq0\}$ generated by any finites set of vectors $v_1,...,v_s$ in $V$. The generators for the cone $\sigma$ are the vectors $v_i$ or positive multiples of $v_i$ called rays. The dimension of the linear space $\mathbb{R}\cdot\sigma=\sigma+(-\sigma)$ spanned by $\sigma$, gives the dimension of $\sigma$ (dim$\sigma$).
	
	For any set $\sigma$ there is a corresponding dual set $\sigma^\smallsmile$, and it is the set of equations of supporting hyperplanes, $\sigma^\smallsmile =\{u\in V^*:\langle u,v\rangle\geq0\ \text{for all}\ v\in\sigma\}$. If $v_0\notin\sigma$ in the case that $\sigma$ is a convex polyhedral cone, then there exists some $u_0\in\sigma^\smallsmile$ where $\langle u_0,v_0\rangle<0$. Therefore $(\sigma^\smallsmile)^\smallsmile=\sigma$ (duality theorem by direct translation of the previous sentence) and other conditions pertaining to the facets also follow.
	
	If $\sigma\cap(-\sigma)=\{0\}$, $\sigma$ does not contain a nonzero linear subspace, $\sigma^\smallsmile$ spans $V^*$, and a $u$ in $\sigma^\smallsmile$ exists such that $\sigma\cap u^\top=\{0\}$, then $\sigma$ is called a strongly convex cone. 
	
	\subsubsection{Affine toric varieties}
	
	For $\sigma$ is a strongly convex rational polyhedral cone, $S_\sigma=\sigma^\smallsmile\cap M$ is a finitely generated semigroup and any additive semigroup $S$ determines a group ring, $\mathbb{C}[S]$, that is a commutative $\mathbb{C}$-algebra. 
	
	As $u$ varies over $S$, the basis of this complex vector space, $\chi^u$, has multiplication given by addition in $S$, i.e. $\chi^u\cdot\chi^{u'}=\chi^{u+u'}$. Where $\chi^0$ is the unit 1. The generators $\{\chi^{u_i}\}$ for the $\mathbb{C}$-algebra are determined by the generators $\{u_i\}$ for the semigroup $S$.
	
	If we use $A$ to denote any finitely generated commutative $\mathbb{C}$-algebra, $A$ determines a complex affine variety Spec$(A)$. If generators of $A$ are specified, then for $I$ an ideal, $A$ is presented as $\mathbb{C}[X_1,\dots,X_m]/I$; and the subvariety $V(I)$ of affine space $\mathbb{C}^m$ of common zeros of the polynomials in $I$ identifies Spec$(A)$. Therefore $A$ is the domain, Spec$(A)$ is an irreducible variety, and even though Spec$(A)$ officially includes all prime ideals of $A$ (corresponding to subvarieties of $V(I)$), a point in Spec$(A)$ is taken as an ordinary closed point corresponding to the maximal ideal, which is denoted Specm$(A)$. Any homomorphism $A\rightarrow B$ determines a morphism of varieties Spec$(B)\rightarrow$ Spec$(A)$, and $\mathbb{C}$-algebra homomorphisms from $A$ to $\mathbb{C}$. When $X=\text{Spec}(A)$ for each element $f\in A$ with $f\neq0$, then according to the localisation homomorphism $A\mapsto A_f$, $X_f$ is the principal open subset, defined $X_f=\text{Spec}(A_f)\subset X=\text{Spec}(A)$.
	
	When $\mathbb{C}=\mathbb{C}^*\cup\{0\}$ is taken as an abelian group via multiplication, and $A=\mathbb{C}[S]$ - constructed from a semigroup, then the points simply correspond to homorphisms of the semigroup from $S$ to $\mathbb{C}$, i.e. $\text{Specm}(\mathbb{C}[S])=\text{Hom}_{sg}(S,\mathbb{C})$.  For a semigroup homomorphism, the value of the corresponding function $\chi^u$ at the specific point of $\text{Specm}(\mathbb{C}[S])$, where $x$ denotes the semigroup homomorphism map from $S$ to $\mathbb{C}$ is the image of $u$ by the map $x$, $\chi^u(x)=x(u)$. And $S=S_\sigma$ is apparent from a strongly convex rational polyhedral cone. We get $A_\sigma=\mathbb{C}[S_\sigma]$ and affine toric variety $U_\sigma=\text{Spec}(\mathbb{C}[S_\sigma])=\text{Spec}(A_\sigma)$ when $S=S_\sigma$ because the polyhedral cone is strongly convex and rational.
	
	For a strongly convex rational polyhedral cone, $S=S_\sigma$, setting $A_\sigma=\mathbb{C}[S_\sigma]$ and $U_\sigma=\text{Spec}(\mathbb{C}[S_\sigma])=\text{Spec}(A_\sigma)$ the corresponding affine toric variety.
	
	Of the group $M=S_{\{0\}}$ all of the above semigroups will be sub-semigroups, so that if $e_1,\dots,e_n$ is a basis for $N$ and $e_1^*,\dots,e_n^*$ is the dual basis of $M$, rename $X_i=\chi^{e_i^*}\in\mathbb{C}[M]$. The generators of semigroup $M$ are $\pm e_1^*,\dots,\pm e_n^*$, and therefore, $\mathbb{C}[M]=\mathbb{C}[X_1,\dots,X_n]_{X_1\cdot\dots\cdot X_n}$ which is a ring of Laurent polynomials in $n$ variables. $\text{Spec}(\mathbb{C}[M])=U_{\{0\}}\simeq\mathbb{C}^*\times\mathbb{C}^*=(\mathbb{C}^*)^n$ is an affine algebraic torus.
	
	All subgroups will be considered to be sub-semigroups of a lattice $M$, therefore $\mathbb{C}[S]$ will be the subalgebra of $\mathbb{C}[M]$, and specifically $\mathbb{C}[S]$ is the domain. Elements of $\mathbb{C}[S]$ are written as Laurent polytnomials in variables $X_i$ when the basis of $M$ is as above. In other words, monomials in the variables $X_i$ generate all those algebras.
	
	The torus $T_N$ that corresponds to $M$ or $N$, can be written $T_N=\text{Spec}(\mathbb{C}[M])=\text{Hom}(M,\mathbb{C}^*)=N\otimes_\mathbb{Z}\mathbb{C}^*$. 
	
	A homomorphism between semigroups $S$ and $S'$ implies a morphism of the affine varieties $\text{Spec}(\mathbb{C}[S'])\rightarrow\text{Spec}(\mathbb{C}[S])$. If $\sigma$ contains $\tau$, taking $u\in S_\sigma=\sigma^\smallsmile\cap M$, $\tau$ is a rational convex polyhedral cone that is equal to $\sigma\cap u^\perp$, and therefore every face, $\tau$ of $\sigma$ can be written $S_\tau=S_\sigma+\mathbb{Z}_{\geq0\cdot(-u)}$. Therefore every basis element of the algebra of $S_\tau$ can be written as $\chi^{w-pu}=\chi^w/(\chi^u)^p$, where $w\in S_\sigma$, and therefore $A_\tau=(A_\sigma)_{\chi^u}$.
	
	It follows that $S_\sigma$ is a subgroup of $S_\tau$, implying the morphism $U_\tau\rightarrow U_\sigma$ which embeds $U_\tau$ as a principal open subset of $U_\sigma$.
	
	A homomorphism $\varphi$ of lattices $N$ and $N'$ so that $\varphi_\mathbb{R}$ maps the cone $\sigma'$ in $N'$ into a cone $\sigma$ in $N$.Its dual map $\varphi'$ maps $S_\sigma$ to $S_\sigma'$ establishing homomorphism $A_\sigma\rightarrow A_{\sigma'}$ and so the morphism from $U_{\sigma'}$ to $U_\sigma$.
	
	Therefore $S_\sigma$ is saturated and $M=S_\sigma+(-S_\sigma)$. For $S_\sigma$ generated by $u_1,\ldots,u_t$, $I$ is generated by polynomials $Y_1^{a_1}\cdot Y_2^{a_2}\cdot\ldots\cdot Y_t^{a_t}-Y_1^{b_1}\cdot Y_2^{b_2}\cdot\ldots\cdot Y_t^{b_t}$ where $a_i, b_i\in\mathbb{R}_{\geq0}$ and $a_1u_1+\ldots a_tu_t=b_1u_1+\ldots b_tu_t$ and so $$A_\sigma = \mathbb{C}[\chi^{u_1},\ldots,\chi^{u_t}] = \mathbb{C}[Y_1,\ldots,Y_t]/I$$.
	
	
	The torus $T_N$ acts on $U_\sigma$, $T_N\times U_\sigma\rightarrow U_\sigma$ for some $\sigma$ in $N$. The product $t\cdot x$ is a map of semigroups $S_\sigma\rightarrow\mathbb{C}$: $u\mapsto t(u)x(u)$. Where $t\in T_N$ corresponds with the map $M\rightarrow\mathbb{C}^*$, and in turn $x\in U_\sigma$ with map $S_\sigma\rightarrow\mathbb{C}$. In fact, for $u\in S_\sigma$ the map $\chi^u$ to $\chi^u\otimes\chi^u$ gives the dual map on algebras $\mathbb{C}[S_\sigma]\rightarrow\mathbb{C}[S_\sigma]\otimes\mathbb{C}[M]$, and is the usual product of groups when $\sigma=\{0\}$, is in agreement with the inclusion of open subsets that are associated with faces of the cone, and expends to the action of $T_N$ on itself.
	
	\subsubsection{Fans and toric varieties}
	
	Henceforth all cones will be rational strongly convex polyhedral cones. And fans are finite. To assemble $X(\Delta)$ from fan $\Delta$, is to take the disjoint union of affine toric varieties $U_\sigma$, one for each $\sigma$ in $\Delta$. A gluing aspect is also performed: $\sigma\cap\tau$ is a face for each cone $\sigma$ and $\tau$. This implies that $U_{\sigma\cap\tau}$ is the principal open subvariety of $U_\sigma$ and $U_\tau$. That $U_{\sigma\cap\tau}\rightarrow U_\sigma\times U_\tau$ is a closed embedding implies the map $A_\sigma\otimes A_\tau\rightarrow A_{\sigma\cap\tau}$ is surjective since $S_{\sigma\cap\tau}=S_\sigma+S_\tau$. Where, of course, $U_{\sigma\cap\tau}=U_\sigma\cap U_\tau$.
	
	Any product of affine and projective spaces can be realised as a toric variety. 
	
	\subsubsection{Toric varieties from polytopes}
	
	Once the vertices of the polytope are established, then the fan $\Delta$ and then the toric variety. 
	A very important method of assembly begins with a rational polytope $P$ in $M_\mathbb{R}$ that is $n$-dimensional and doesn't necessarily contain the origin. The corresponding fan $\Delta_P$ is assembled by taking, for each face $Q$ of $P$ its cone $\sigma_Q=\{v\in N_\mathbb{R}:\langle u,v\rangle\leq\langle u',v\rangle\ \text{for all}\ u\in Q\ \text{and}\ u'\in P\}$.

	\subsection{Contact Manifold}
	
	Recalling that Toric varieties are geometric objects defined by combinatorial information: complex contact manifolds arise naturally in differential geometry, algebraic geometry and exterior differential systems. The geometry of such a manifold is governed by the contact lines contained in it. These are related to the notion of a variety of minimal rational tangents.
	
	\begin{definition}
		A contact structure on a $(2n+1)$-dimensional differentiable manifold $M$ is a smooth codimension 1 distribution (i.e. tangent hyperplane field) $\xi\in TM$ that is maximally non-integrable in the following sense: We require that $\xi$ be a locally defined as $\xi=\text{ker}\alpha$ with a differential 1-form $\alpha$ that satisfies $\alpha\wedge(d\alpha)^n\neq0$ on its whole domain of definition, If such a 1-form exists globally on $M$, it is called a contact form on $M$. A manifold with a contact structure is called a contact manifold.
	\end{definition}
	
	\begin{exmp}
		The standard contact structure on $\mathbb{R}^{2n+1}$ is the kernel of $\alpha=dz-\sum_{i=1}^ny_idx_i$ for coordinates $(x_1,y_1,\ldots,x_n,y_n,z)\in\mathbb{R}^{2n+1}$. 
		If we look at the standard contact structure in $\mathbb{R}^3$ $(\mathbb{R}^3,\text{ker}(dz+xdy))$, the conditions would be that $\alpha=dz+xdy$, so $d\alpha=dx\wedge dy$ and certainly $\alpha\wedge d\alpha\neq0$. Plotting this contact plane, the 1-form $\alpha$ is always tangent to the vector $\frac{\partial}{\partial x}\in\text{ker}(\alpha)$; and this is further improved by again looking at the kernel, giving us $x\frac{\partial}{\partial z}-\frac{\partial}{\partial y}$, so it is spanned by these two vectors $\frac{\partial}{\partial x}$ and $x\frac{\partial}{\partial z}-\frac{\partial}{\partial y}$.	
	\end{exmp}

	\begin{theorem} [Darboux's theorem] Let $\alpha$ be a contact form on the $(2n+1)$-dimensional manifold $M$ and $p$ a point in $M$. Then there are coordinates $(x_1,\ldots,x_n,y_1,\ldots,y_n,z)$ on a neighbourhood $U\subset M$ of $p$ such that $$\alpha|_U=dz+\sum_{j=1}^nx_jdy_j$$.
	\end{theorem}
	
	One can easily find the respective duals $v_i^*$ to $v_1$ and $v_2$ that construct the dual vector space $M$ with respect to the dual pairing $\langle,\rangle$ (as an example one can find the duals corresponding to the standard contact structure). And therefore one can create a variety of minimal rational tangents.
	
	\begin{definition} Let $X$ be a normal variety, $U=X-\text{Sing}X$ and $j:U\rightarrow X$ the natural embedding; the sheaf of differential p-forms or of p-differentials (in the sense of Zariski-Steenbrink) to be $\Omega^p_X=j_*(\Omega^p_X)$.
	\end{definition}
	
	
	Let $\mathfrak{U}$ denote the ideal generated by all monomials defined before. The sheaf $\Omega_X^p$ on the toric variety $X$ corresponds to a certain $\mathbb{C}[M]$-module. $\Omega_A^p$ is the $N$-graded $\mathbb{C}$ vector space $\Omega^p_A$ as defined in \cite{Danilov:1978dg} is isomorphic to $\mathfrak{U}$.
	
	\section{proof of theorem}
	
	A \emph{contact structure} on a smooth algebraic variety is the date of a subbundle $D\subset\mathcal{T}_X$ of rank $\text{dim}X-1$ such that the $\mathcal{O}_X$-bilinear form on $D$ with value in the line bundle $L=\mathcal{T}_X/D$ induced by the Lie bracket on $\mathcal{T}_X$ is non degenerate in each points of $X$. This implies that $X$ has odd dimension $2n+1$ and that the canonical bundle $K_X$- is isomorphic to $L^{-1-n}$. We can also define the contact structure by the data of an element $\theta\in H^0(X,\Omega_X^1\otimes L)$, the \emph{contact form}, such that $\theta\wedge(d\theta)^n$ is everywhere nonzero.
	
	Complex projective spaces and varieties $\mathbb{P}_Y(\mathcal{T}_Y)$, with $Y$ a smooth variety, are examples of contact varieties. 
	
	In short, a \emph{toric variety} is an equivariant compactification of a torus \cite{Danilov:1978dg}, \cite{Fulton:1993dg}, \cite{Oda:1978dg}, \cite{Oda:1988dg}.
	
	The study of contact varieties happens to be difficult. In dimension 3 they are classified by using Mori theory. Other authors study contact structures on Fano varieties but the results are not complete yet.
	
	\begin{theorem} \cite{Druel:1999dg}
		Let X be a smooth projective toric variety with dimension $2n+1$ $(n\geq 1)$ endowed with a contact structure.
		Then $X$ is isomorphic either to the complex projective space $\mathbb{P}^{2n+1}$ or to the variety $\mathbb{P}_{\mathbb{P}^1\times\cdots\times\mathbb{P}^1}(\mathcal{T}_{\mathbb{P}^1\times\cdots\times\mathbb{P}^1})$
	\end{theorem}
	
	\begin{proof}\renewcommand{\qedsymbol}{}
		The intersection pairing between 1-cycles and divisors induces a duality between the following two real vector spaces:
		$$N_1(X)=(\{1-\text{cycles}\}/\equiv)\otimes\mathbb{R}\ \text{and}\ N^1(X)=(\{1-\text{divisors}\}/\equiv)\otimes\mathbb{R},$$
		with $\equiv$ the numerical equivalence. As stated before, dimension of both vector spaces is called the Picard number of $X$.
		If we consider the cone as $NE(X)$, and let it be spanned by classes of effective 1-cycles, $NE(X)\subset N_1(X)$. Then an extremal ray is an half-line $R$ in $\bar{NE}(X)$ (the closure of $NE(X)$ in $N_1(X)$), such that $K_X.R^*<0$ and such that for all $Z_1,Z_2\in\bar{NE}(X)$ if $Z_1+Z_2\in R$ then $Z_1,Z_2\in R$ \cite{Mori:1982dg}. An extremal rational curve is an irreducible rational curve such that $\mathbb{R}^+[\mathbb{C}]$ is an extremal ray and $-K_X.C\leq \text{dim} X+1$. The first result from Mori theory is that every extremal ray is spanned by an extremal rational curve. The second fundamental result is that any extremal ray admits a contraction, i.e. there exists a normal projective variety $Y$ and a morphism $\phi:X\rightarrow Y$, surjective with connected fibers, contracting irreducible curves $C$ such that $[C]\in R$ (Kawamata-Shokurov theorem).
		
		For toric varieties, there is a more precise result \cite{Reid:1983}. Let $X$ be a smooth projective toric variety associated to the $\Delta$. We use notations from \cite{Oda:1988dg}. There there exist elements $\tau_1,\ldots,\tau_s\in\Delta(d-1)$ such that
		$$ NE(X)=\bar{NE}(X)=\mathbb{R}^+[V(\tau_1)]+\ldots+\mathbb{R}^+[V(\tau_s)]. $$
		The half lines $\mathbb{R}^*[V(\tau_i)]$ are called generalised extremal rays. Let $R\subset NE(X)$ a generalised extremal ray. Then there exists a projective toric variety $Y$ and an equivariant morphism $\phi:X\rightarrow Y$ which is a contraction in Mori's sense. Denote by $A\subset X$ the exceptional locus $\phi$ and $B=\phi(X)$:
		\begin{center}	
			\begin{tikzpicture}
				
				\node (A) at (-1,0) {$X$};
				\node (B) at (1,0) {$Y$};
				\node (C) at (-1,-1) {$\cup$};
				\node (D) at (1,-1) {$\cup$};
				\node (E) at (-1,-2) {$A$};
				\node (F) at (1,-2) {$B$};
				
				\draw [thick, ->] (A) -- (B) node[midway,above] {$\phi$};
				\draw [thick, ->] (E) -- (F) node[midway,above] {$\phi|_A$};
			\end{tikzpicture}
		\end{center}
		Then $A$ and $B$ are two irreducible toric stratas and the morphism $\phi_A:A\rightarrow B$ is flat and its fibers are weighted projective spaces. Finally, when the contraction is of fibered type, the variety $Y$ is regular and the morphism $\phi$ is smooth.
		
		Finally we recall a general result from J. Wisniewski \cite{Wisniewski:1989dg}, \cite{Wisniewski:1991dg} about the exceptional locus of extremal contraction. Let $F$ be an irreducible component of a non-trivial fiber of an elementary contraction associated to the extremal ray $R$. We call the locus of $R$, the locus of curves with numerical equivalence class in $R$. Then we have the inequality:
		$$\text{dim}F+\text{dim}(\text{locus of}\ R)\geq\text{dim} X+l(R)-1$$
		where $l(R)$ is the length of the extremal ray $R$:
		$$l(R)=\text{inf}\{-K_X.C_0|\ \text{with}\ C_0 \text{a rational curve and} C_0\in R\}.$$

		The rest of the proof of this theorem relies on the study of extremal contractions of $X$. Denote by $T$ the torus acting on $X$. As the canonical divisor $K_X$ of $X$ is not effective, there exists an extremal ray $R$ in the sense of Mori spanned by $C$ a smooth, rational, $T$-invariant curve. As $K_X=-(n+1)L$, we have $l(R)=n+1$ or $l(R)=2n+2$ with $l(R)$ the length of the extremal ray $R$. For the last case, $X$ is Fano and $Pic(X)=\mathbb{Z}$ (\cite{Wisniewski:1989dg} Prop. 2.4). Then $X$ is isomorphic to the complex projective space $\mathbb{P}^{2n+1}$ (\cite{Oda:1978dg} Thm. 7.1).
		
		It remains to study the case $l(R)=n+1$. Denote by $\phi:X\rightarrow Y$ the Mori contraction associated to the ray $R$, with $Y$ a projective toric variety. Denote by $T'$ the torus acting on $Y$. Denote by $A\subset X$ the exceptional locus of $\phi$ and $B=\phi(X)$:
		\begin{center}
			\begin{tikzpicture}
				\node (A) at (-1,0) {$X$};
				\node (B) at (1,0) {$Y$};
				\node (C) at (-1,-1) {$\cup$};
				\node (D) at (1,-1) {$\cup$};
				\node (E) at (-1,-2) {$A$};
				\node (F) at (1,-2) {$B$};
				
				\draw [thick, ->] (A) -- (B) node[midway,above] {$\phi$};
				\draw [thick, ->] (E) -- (F) node[midway,above] {$\phi|_A$};
			\end{tikzpicture}
		\end{center}
		As the curve $C$ is $T$-invariant, it is contracted to a fixed point $y\in B$ under $T'$. The morphism $\phi$ being equivariant, the fiber $F=\phi^{-1}(y)$ is $T$-invariant. Hence, as this fiber is irreducible, we deduce that $F$ is a toric strata and that $G$ is smooth as $X$ is not. Then the fiber $F$ is isomorphic to projective space $\mathbb{P}^k$ $(k\geq1)$ and, as $C$ is a toric strata of the toric variety $\mathbb{P}^k$, the curve $C$ is identified to a line in $\mathbb{P}^k$. We know that there exists a rational curve $C_1\subset X$ such that $C_1\in R$ and $-K_X.C_1=n+1$. But, as there exists a divisor $D$ on $X$ such that $D.C=1$ and as $R=\mathbb{R}^+[C]$, we easily deduce that $C\equiv C_1$ hence $C.L=1$. Therefore we have $(F,L_{|F})=(\mathbb{P}^k,\mathcal{O}_{\mathbb{P}^k}(1))$.
		
		As $H^0(\mathbb{P}^k,\Omega^1_{\mathbb{P}^k}(1))=(0)$, the restriction of the contact form $\theta$ to $\mathbb{P}^k$ vanishes and by a computation with local coordinates, we check that for all $x\in\mathbb{P}^k$, the vector subspace $\mathcal{T}_{\mathbb{P}^k}(x)\subset D(x)$ is totally isotropic for the alternate contact form which is non-degenerate by hypothesis. This implies the inequality $k\leq n$ and the equality $k=n$ as $\text{dim}F\geq l(R)-1=n$. Applying again Wisniewski inequality, we deduce that the contraction is of fibered type and therefore $Y$ is regular, $n+1$-kdimensional and $\phi$ is smooth. 
		
		Consider the exact sequence:
		$$ 0\rightarrow\mathcal{T}_{X/Y}\rightarrow\mathcal{T}_X\rightarrow\phi^*\mathcal{T}_Y\rightarrow0 $$
		The map $\mathcal{T}_{X/Y}\rightarrow L$ obtained by composition with the projection $\mathcal{T}_X\rightarrow L$ being identically zero by Grauert theorem and Lemma 3.1 and Lemma 3.2 in \cite{Druel:1999dg}, there exists a surjective map $\phi^*\mathcal{T}_Y\rightarrow L\rightarrow0$ and therefore a morphism $X\rightarrow \mathbb{P}_Y(\mathcal{T}_Y)$ over $Y$ which induces an isomorphism on each fiber. Then this morphism is in fact an isomorphism. Therefore the result follows from\\ 
		(1) the fact that for $Y$ a smooth projective toric variety of dimension $n$ and $\mathcal{E}$ a vector bundle of rank $r+1$ over $Y$. And with the assumption that $X=\mathbb{P}_Y(\mathcal{E})$ is a toric variety and the natural morphism $\phi:X\rightarrow Y$ is equivariant. Then the bundle $\mathcal{E}$ is the direct sum of its isotypical components.\\
		And (2) for $X$ a smooth projective toric variety of dimension $n\geq 1$ with a totally reducible tangent bundle. Then $X$ is isomorphic to $\mathbb{P}^1\times\cdots\times\mathbb{P}^1$.
	\end{proof}

	\begin{theorem}
		For $X$ a smooth projective toric variety of dimension $2n+1$ $(n\geq 1)$, defined over the field $\mathbb{C}$ of complex numbers and endowed with a contact structure. Then either $X$ is isomorphic to the complex projective space $\mathbb{P}^{2n+1}$ or $X$ is isomorphic to the variety $\mathbb{P}_{\mathbb{P}^1\times\cdots\times\mathbb{P}^1}(\mathcal{T}_{\mathbb{P}^1\times\cdots\times\mathbb{P}^1})$. The momentum map is defined $$\mu(x)=\frac{1}{\sum|\chi^u(x)|}\sum\limits_{u\in P\cap M}|\chi^u(x)|u$$ where $\chi^u$ for $u\in P\cap M$ are the sections of $\varphi:X\rightarrow\mathbb{P}^{2n}$ which is convex.
	\end{theorem}
	
	\begin{proof}\renewcommand{\qedsymbol}{}
		For X a smooth projective toric variety with dimension $2n+1$ $(n\geq 1)$ endowed with a contact structure.
		$X$ is then isomorphic either to the complex projective space $\mathbb{P}^{2n+1}$ or to the variety $\mathbb{P}_{\mathbb{P}^1\times\cdots\times\mathbb{P}^1}(\mathcal{T}_{\mathbb{P}^1\times\cdots\times\mathbb{P}^1})$.
		
		Since the toric varieties have been established, what is left is to talk about how a toric variety will be redefined in the event of singularities, and according to compactness. This is before fully identifying the orbits, topology and line bundles all required in order to formulate the momentum polytope.
		
		\underline{The considerations for singularities and compactness}
		
		For specific point $x_\sigma$ of its affine variety, $S_\sigma=\sigma^\smallsmile\rightarrow (1,0)\subset\mathbb{C}^*\cup\{0\}=\mathbb{C}$ where $u$ maps to 1 if $u\in\sigma^\perp$ and 0 otherwise. The existence of fixed points in $U_\sigma$ is contingent on $\sigma$ spanning $N_\mathbb{R}$, and in turn $x_\sigma$ is then a unique fixed point of $T_N$ acting on $U_\sigma$. For singular points of the toric varieties, suppose $U_\sigma$ is nonsingular at point $x_\sigma$, which means that $\sigma^\smallsmile$ can't have more than $n$ edges. And eventually, $U_\sigma$ is isomorphic to $\mathbb{C}^n$. If we let $N_\sigma=N\cap\sigma+(-\sigma\cap N)$ be a sublattice of $N$, then $N(\sigma)=N/N_\sigma$ is the quotient group that is also a lattice, and the map $N_\sigma\rightarrow N\rightarrow N(\sigma)$ . And so a cone is nonsingular if it is generated by part of a basis for the lattice. And if all the cones of a fan are nonsingular, then in turn that fan is nonsingular. Although a toric variety can be singular, every toric variety is normal. $\sigma$ is generated by part of a basis of $N$ if affine toric variety $U_\sigma\simeq\mathbb{C}^k\times(\mathbb{C}^*)^{n-k}$ is nonsingular, where $\text{dim}(\sigma)=k$. And each ring $A_\sigma=\mathbb{C}[S_\sigma]$ is integrally closed. Therefore we have a method for extrapolating a nonsingular case from a singular case.
		
		For a lattice $N$ of any rank, a sublattice $N'\subset N$ of finite index, with corresponding duals $M\subset M'$, then $M'/M\times N/N'\rightarrow\mathbb{Q}/\mathbb{Z}$ by $\langle,\rangle$ pairing, and $\mathbb{Q}/\mathbb{Z}\hookrightarrow\mathbb{C}^*$ by $q\mapsto\text{exp}(2\pi iq)$. Therefore $G=N/N'$ acts on the algebra $\mathbb{C}[M']$ by $v\cdot\chi^{u'}=\text{exp}(2\pi i\langle u',v\rangle)\cdot\chi^{u'}$ for $u'\in M'$ and $v\in N$ and therefore $\mathbb{C}[M']^G=\mathbb{C}[M]$. And $T_{N'}/G=T_N$.
		
		For $\sigma$ any simplex, meaning that it is generated by independent vectors, for example if $\sigma$ is an $n$-simplex and therefore generated by $n$ independent vectors, therefore the minimal elements in $\sigma\cap N$ along all the $n$ edges generate the sublattice $N'$ giving a cone $\sigma'$, $\mathbb{C}^n=U_{\sigma'}\rightarrow U_\sigma$. And we have the result of action $U_{\sigma'}/G=U_\sigma=\mathbb{C}^n/G$. Now if $\sigma$ is any simplex, and in particular if all the cones in $\Delta$ are simplicial, making it in turn a simplicial fan, then the toric variety only has quotient singularities (is an orbifold).
		
		For non-affine toric varieties, starting with the same fan used in the assembly of the projective space, meaning this fan's cones are generated by proper subsets of vectors $\{v_0,v_1,\ldots,v_n\}$ of which any $n$ linearly independent vectors exist and whose sum is zero. Whilst $N$ is generated by vectors $(1/d_i)\cdot v_i$, $0\leq i\leq n$. And so we have a 'twisted' or 'weighted projective space', $\mathbb{P}(d_0,\ldots,d_n)=\mathbb{C}^{n+1}\backslash\{0\}/\mathbb{C}^*$, where the $\mathbb{C}^*$ action is defined, $\varsigma\cdot(x_0,\ldots,x_n)=(\varsigma^{d_0}x_0,\ldots,\varsigma^{d_n}x_n)$. This therefore makes the group action compatible on affine open subvarieties.
		
		Therefore one can find a sublattice for any cone in a simplicial fan, for which the associated open subvariety of the toric variety can be expressed as a quotient by any finite abelian group. Like for the weighted projective spaces, there usually exists a sublattice for all the open sets, but not always. There exist fans with particular edges for which every sublattice, the variety is singular.
		
		If we take $\mathbb{G}_m$ to be the multiplicative algebraic group for, say, $\mathbb{C}^*$, then every one-parameter subgroup $\lambda:\mathbb{G}\rightarrow T_N$ exists for a unique $v\in N$, denoted $\lambda_v$. For $z\in\mathbb{C}^*$ $\lambda_v(z)\in T_N$ meaning it is given by a group homomorphism from $M$ to $\mathbb{C}^*$. Where explicitly for $u$ in $M$, $\lambda_v(z)(u)=\chi^u(\lambda_v(z))=z^{\langle u,v\rangle}$ and $\langle,\rangle$ is the dual pairing $M\otimes N\rightarrow \mathbb{Z}$. Where the dual $M=\text{Hom}(N,\mathbb{Z})=\text{Hom}(T_N,\mathbb{G}_m)$, and every character $\chi:T_N\rightarrow\mathbb{G}_m$ is given by a unique $u$ in $M$. The function $\chi^u$ in the coordinate ring $\mathbb{C}[M]=\Gamma(T_N,\mathcal{O}^*)$ gives the character corresponding to $u$.
		This gives how to ascertain $N$ from $T_N$.
		
		Looking at the limits $\lim_{z\to0} \lambda_v(z)$ for different values of $v\in N$, then this way given a $\sigma$ we can recover $\sigma$ from the $T_N\subset U_\sigma$ torus embedding.
		
		If $|\Delta|$ contains $v$, and $\Delta$ has cone $\tau$ that contains $v$ in its relative interior then $x_\tau=\text{lim}_{z\rightarrow0}\lambda_v(z)$. And if $v$ is not in any cone of $\Delta$, then $\text{lim}_{z\rightarrow0}\lambda_v(z)$ does not exist in $X(\Delta)$. There is no limit or converging subsequence for $v$s not in the union of cones $\Delta$ (not in the support $|\Delta|$).

		$X(\Delta)$ is compact if and only if $|\Delta|=N_\mathbb{R}$.
		For a homomorphism $\phi:N'\rightarrow N$, which maps fan $\Delta'$ to $\Delta$, the morphism $\phi_*:X(\Delta')\rightarrow X(\Delta)$ is proper if and only if $|\Delta'|=\phi^{-1}(|\Delta|)$. An example of a proper map is what's called blowing up. 
		
		Two dimensional nonsingular complete toric varieties are given by sequence of lattice points, $v_0,v_1,\ldots,v_{d-1},v_d=v_0$, in $N=\mathbb{Z}^2$, so that successive pairs generate the lattice. Since $v_0$ and $v_1$ are a basis for a lattice, it follows that $v_1$ and $v_2$ are also a basis where $v_2=-v_0+a_1v_1$ for $a_1\in\mathbb{Z}$ (i.e. $a_iv_i=v_{i-1}+v_{i-1}$ where $1\leq i\leq d$ for some integers $a_i$). Cones created (for example between $v_{i-1}$ and $v_{i+1}$) cannot have a third vector (e.g. $v_j$) in the angle strictly between them. This is an example of how the possible configurations are topologically constrained. All of these surfaces can be classified. In fact, all toric surfaces that are complete and nonsingular are procured from $\mathbb{P}^2$ or $\mathbb{F}_a$ by a series of blow-ups at $T_N$-fixed points. 
		
		A strongly convex cone is generated by $v_{j-1}$ and $v_{j+1}$ for some $j$ $1\leq j\leq d$ when $d\geq5$. And we have that $v_j=v_{j-1}+v_{j+1}$.
		
		Before where we fixed the cone to refine the lattice, in order to resolve singularities, let us now look at fixing the lattice and subdividing the cones. A refinement $\Delta'$ can be created from $\Delta$, so each cone of $\Delta$ is a union of cones in $\Delta'$. We obtain a birational and proper morphism $X(\Delta')\rightarrow X(\Delta)$ from the identity map of $N$. This map is an isomorphism on $T_N$ contained in each. More details can be found in section \textsection 2.6 of \cite{Fulton:1993dg}, including an example of creating new generators of the cone using the Hirzebruch-Jung continued fraction of $m/k$ so that, for example, $v=e_2$ and $v'=me_1-ke_2$ generate $\sigma$ rather than $e_1$ and $e_2$ for a cone that is not generated by a basis for $N$. The indicated subdivision provides a nonsingular $X(\Delta')$ mapping birationally and properly to $U_\sigma$. In similar such ways this construction can be used on singular toric varieties to resolve singularities.
		
		In more detail, given a fan $\Delta$ in any lattice $N$ and any lattice point $v$ in $N$, then $\Delta$ can be created from subdividing $\Delta$. Meaning that each cone that contains $v$ is exchanged with the sums of faces that have a ray through $v$, and every cone that doesn't contain $v$ is left unchanged. The mapping from $X(\Delta')$ to $X(\Delta)$ is proper and birational since $\Delta$ and $\Delta'$ have the same support. Such subdivisions are taken until a nonsingular toric variety is reached. 
		
		
		\underline{With accordance to orbits, topology and line bundles}
		
		So we have established that $X=X(\Delta)=\bigcup_{disjoint}$ its orbits by $T_N$-action. There exists an orbit $O_\tau$ for each cone $\tau$ in $\Delta$ containing distinguished point $x_\tau$, and if $\text{dim}(\tau)=k$ then $O_\tau\simeq(\mathbb{C}^*)^{n-k}$; if $x_\tau $ is $n$-dimensional then $O_\tau$ is $x_\tau$; and if $\tau=\{0\}$, then $O_\tau=T_N$. Denote by $V(\tau)$ the open subvariety of the closure of $O_\tau$. It is the closed subvariety of the toric variety $X$ and $V(\tau)=\bigcup_{disjoint} O_\gamma$ where $\tau$ is contained in $\gamma$ as a face.
		
		In general $O_\tau=\text{Hom}(\tau^\perp\cap M,\mathbb{C}^*)$. Let $N_\tau$ denote the sublattice of $N$ generated by $\tau\cap N$ and define the quotient lattice and its dual as $N(\tau)=N/N_\tau$, and $M(\tau)=\tau^\perp\cap M$. Therefore $T_{N(\tau)}=\text{Hom}(M(\tau),\mathbb{C}^*)=O_\tau$ is the torus corresponding to these lattices, of dimension $n=k$ where $k$ is the value of the dimension of $\tau$. And we have that $O_\tau=\text{Spec}(\mathbb{C}[M(\tau)])=N(\tau)\otimes_\mathbb{Z}\mathbb{C}^*$, $T_N$ acts transitively via the projection $T_N\rightarrow T_{N(\tau)}$.
		
		Next we will introduce something called the 'Star of the cone $\tau$' which can be loosely defined as the set of cones $\sigma$ in $\Delta$ that contain $\tau$ as a face. The cones $\bar{\sigma}=\sigma+(N_\tau)_\mathbb{R}/(N_\tau)_\mathbb{R}$ where $\bar{\sigma}$ is such that $\tau<\sigma$, is a subset of $N_\mathbb{R}/(N_\tau)_\mathbb{R}=N(\tau)_\mathbb{R}$. Forming a fan in $N(\tau)$, denoted $\text{Star}(\tau)$. Conceptualise $\text{Star}(\tau)$ as the cones that contain $\tau$ but register it as a fan in $N(\tau)$.
		
		The $(n-k)$-dimensional toric variety $V(\tau)$ is therefore defined as $X(\text{Star}(\tau))$.
		
		As $\sigma$ varies over all the cones in $\Delta$, the corresponding toric variety $V(\tau)$ has an affine open covering $U_\sigma(\tau)$, with $\tau$ as a face. This affine open covering is defined as $\text{Spec}(\mathbb{C}[\sigma^\smallsmile\cap\tau^\perp\cap M])$ where the corresponding dual face to $\tau$ is $\sigma^\smallsmile\cap\tau^\perp$, a face of $\sigma^\smallsmile$. 
		
		$U_\sigma(\tau)$ is embedded in $U_\sigma$, this can be seen due to the respective semigroup homomorphisms: $\text{Hom}_{sg}(\sigma^\smallsmile\cap\tau^\perp\cap M,\mathbb{C})$ is embedded in $\text{Hom}_{sg}(\sigma^\smallsmile\cap M,\mathbb{C})$. Since the extension by zero of a semigroup homomorphism is a semigroup homomorphism as $\sigma^\smallsmile\cap\tau^\perp$ is a face of $\sigma^\smallsmile$. In fact if $u$ is in $\sigma^\smallsmile\cap\tau^\perp\cap M$, then $\chi^u$ is projected to $\chi^u$, by the surjection of rings $\mathbb{C}[\sigma^\smallsmile\cap M]\twoheadrightarrow\mathbb{C}[\sigma^\smallsmile\cap\tau^\perp\cap M]$. And we therefore have the closed embedding $$V(\tau)\hookrightarrow X(\Delta).$$ And closed embeddings $V(\tau')\hookrightarrow V(\tau)$ if $\tau$ is a face of $\tau'$, given by extension by zero from $\text{Hom}_{sg}(\sigma^\smallsmile\cap\tau'^\perp\cap M,\mathbb{C})$ to $\text{Hom}_{sg}(\sigma^\perp\cap\tau^\perp\cap M,\mathbb{C})$, on $U_\sigma$ for $\sigma\in\text{Star}(\tau')$. Or also given by considering $V(\tau)$ as a toric variety and exerting the construction described above. Therefore establishing an order-reversing correspondence from $\tau$ in $\Delta$ to closures $V(\tau)$ in $X(\Delta)$.
		
		This all implies that $\oplus\mathbb{C}\cdot\chi^u$ is the ideal of $V(\tau)\cap U_\sigma$ in $A_\sigma$. This ideal is the sum over all $u$ in $S_\sigma$ so that for $v$ in the relative interior of $\tau$, $\langle u,v\rangle>0$.
		
		The following are the presiding relations between affine open sets $U_\sigma$, orbits $O_\tau$ and orbit closures $V(\tau)$: $U_\sigma = \coprod_{\tau\prec\sigma}O_\tau$; $V(\tau)=\coprod_{\gamma\succ\tau}O_\gamma$; $O_\tau=V(\tau)\setminus\bigcup_{\gamma\succneqq\tau}V(\gamma)$.
		
		According to the third relation, $T_N=X(\Delta)\setminus\bigcup_{\gamma\neq{0}}V(\gamma)$ for $\tau=\{0\}$. And therefore, $X(\Delta)$ is the disjoint union of the $O_\tau$ that are the orbits of the $T_N$-action. And when $\sigma$ is a face of $\tau$, is precisely when $O_\tau$ is contained within the closure of $O_\sigma$. And the closed orbits are equal to $O_\sigma$ for which $\sigma$ is maximal cone in $\Delta$. Removing some maximal cones $\sigma$ from the fan $\Delta$, whilst maintaining their faces - denote the resulting fan, $\Delta^\circ$ - then $X(\Delta^\circ)$, the toric variety obtained from $X(\Delta)$ with its closed orbits $O_\sigma$ removed.
		
		If we construct a fan from a polytope $P$, then for each face $Q$ of the polytope has a cone $\sigma_Q$, and $V(\sigma_Q)$, the invariant subvariety for each face, can simply be denoted $V_Q$ (when $Q$ is a face of $Q'$, then $V_{Q'}$ contains $V_Q$). And the real dimension of $Q$ is equal to the complex dimension of $V_Q$.
		
		Each $V(\tau)$ is nonsingular if $X(\Delta)$. 
		
		Let us introduce $\pi_1(X(\Delta))$ as the fundamental group of a toric group. Complete toric varieties are simply connected (where $\Delta$ is a fan that includes an $n$-dimensional cone.)(The base points can be considered as the origins of the embedded tori, and so will be omitted). Due to the inclusion $T_N\hookrightarrow X(\Delta)$ provides the surjection $\pi_1(T_N)\twoheadrightarrow\pi_1(X(\Delta))$.
		
		The normality is crucial: the closed complement isn't small enough. Therefore the variety is not simply connected but still have that the complement of the point in the irreducible variety is simply connected. Consider as an example, the variety that is equal to 2 simply connected points.
		
		If we take $v\in N$ to $S^1\subset\mathbb{C}^*\rightarrow T_N$ where $\lambda_v$ is the $\mathbb{C}^*\rightarrow T_N$ map defined before, we have the canonical isomorphism between $N$ and $\pi_1(T_N)$, for any $T_N$. The loop $S^1$ can be contracted to $U_\sigma$ if $v$ is in $\sigma\cap N$ for some cone $\sigma$. As $\text{lim}_{z\rightarrow0}\lambda_v(z)=x_\sigma$ exists in $U_\sigma$, $\lambda_v$ enhances as a $\mathbb{C}\rightarrow U_\sigma$ map. For $z\in S^1$ this contraction is defined as $\lambda_v(tz)$  and $0<t\leq1$, and $x_\sigma$ for $t=0$. For all $t$ the general all-encompassing one-parameter subgroup for $v$ is denoted $\lambda_{v,t}(z)$. 
		
		It follows that if $\sigma'$ is a cone in lattice $N_\sigma$ generated by $\sigma$ then a canonical isomorphism from $\pi_1(U_\sigma)$ to $\pi_1(T_{N(\sigma)})$ which is identically $N(\sigma)$ according to the fibration $U_{\sigma'}\rightarrow U_\sigma\rightarrow T_{N(\sigma)}$. Of course as $\pi_1(\mathbb{C}^*)=\pi_1(S^1)=\mathbb{Z}$ then $U_\sigma=U_{\sigma'}\times(\mathbb{C}^*)^{n-k}$, then $\pi_1(U_\sigma)\simeq\mathbb{Z}^{n-k}$ when $\sigma$ is a $k$-dimensional cone. When $\sigma$ is $n$-dimensional, then $N$ is generated as a group by $\sigma\cap N$. And if such loops are trivial in $U_\sigma$ then all loops are trivial.
		
		If $N'\subset N$ given by $\sigma\cap N$ as $\sigma$ varies over $\Delta$, and since for each $\sigma$ $\pi_1(U_\sigma)=N/N_\sigma$, this means that as $\sigma$ varies over $\Delta$, $\pi_1(X(\Delta))=N/N'$.
		
		And indeed $U_\sigma$ is contractible if $\sigma$ is $n$-dimensional, and $O_\sigma\subset U_\sigma$ is a deformation contract if $\text{dim}(\sigma)=k$. And where $\sigma^\perp\cap M=M(\sigma)$ then if we define a homotopy $H:U_\sigma\times[0,1]\rightarrow U_\sigma$ between the identity map and the retraction $r:U_\sigma\rightarrow x_\sigma$, the canonical isomorphism $H^i(U_\sigma;\mathbb{Z})\simeq\bigwedge^i(M(\sigma))$.
		
		Some information about the cohomology of $X(\Delta)$ can be provided from knowing the cohomology of basic open sets $U_\sigma$, the cohomology of $X$ can be derived as the \u{C}ech cohomology of the covering. 
		
		Looking at the spectral sequence 
		$$E_1^{p,q}=\bigotimes_{i_0<\ldots<i_p}H^q(U_{i_0}\cap\ldots\cap U_{i_p})\Rightarrow H^{p+q}(X).$$ This is the spectral sequence of the double complex $\bigoplus\mathcal{C}^q(U_{i_0}\cap\ldots\cap U_{i_p})$
		if $\mathcal{C}^*$ is a resolution that is injective of the sheaf. And this is true with any sheaf or coefficient group. It is the spectral sequence with vertical maps coming from $\mathcal{C}^*$ (complex) and the horizontal maps the so-called \u{C}ech maps according to alternating sums of restrictions.
		
		Where $\sigma_i$ are the maximal cones of fan $\Delta$, can apply the spectral sequence to the covering according to open sets $U_i=U_{\sigma_i}$, then, $$E_1^{p,q}=\bigotimes_{i_0<\ldots<i_p}\bigwedge^qM(\sigma_{i_0}\cap\ldots\cap\sigma_{i_p})\Rightarrow H^{p+q}(X(\Delta))$$ which gives us $\chi(X(\Delta))$, that is the topological Euler characteristic. For every $\text{dim}\tau<n$, then the alternating sum $\sum(-1)^q\text{rank}(\wedge^qM(\tau))=0$, otherwise when $\text{dim}\tau=n$ then $\sum(-1)^q\text{rank}(\wedge^qM(\tau))=1$ and so, the topological Euler characteristic is $\chi(X(\Delta))=\sum(-1)^q\text{rank}E_1^{p,q}=\sharp n-\text{dimensional cones in} \Delta$. If $\Delta$ is complete i.e. $\text{dim}\Delta=n$, and $U_{\sigma_i}$ is contractible, $E_1^{0,q}=0$ for $q\geq 1$. Additionally, $E^{\cdot,0}_1$ is the cochain complex of a simplex, i.e. a Koszul complex. Meaning that for $p$ greater than $1$, $E_2^{p,0}=0$. Then according to the spectral sequence, $\text{Ker}(E_1^{1,1}\rightarrow E_1^{2,1})=E_\infty^{1,1}=E_2^{1,1}=H^2(X(\Delta))=\text{Ker}(\bigoplus_{i<j}M(\sigma_i\cap\sigma_j)\rightarrow\bigoplus_{i<j<k}M(\sigma_i\cap\sigma_j\cap\sigma_k))$. Where $\chi^u\neq0$ on $U_\sigma$ given by any element $u$ of $M(\sigma)=\sigma^\perp\cap M$. The line bundle on $X(\Delta)$ is given by the cocycle defined by the above kernel, and every element of the first Chern class of a line bundle is $H^2(X)$.
		
		On any variety, a finite formal sum $\sum a_iV_i$ of irreducible closed subvarieties $V$ of codimension one in $X$ is called a Weil divisor. Let $D$ denote the Cartier divisor, which is given by the covering data of $X$ by $U_\alpha$ and $f_\alpha$ -the affine open sets and nonzero rational functions respectively. $f_\alpha$ are referred to as local equations so that the ratios $f_\alpha/f_\beta\neq0$ everywhere are are regular functions on $U_\alpha\cap U_\beta$. If we look at the ideal sheaf of $D$, denoted $\mathcal{O}(-D)$, it is a subsheaf of the sheaf of rational functions generated by $f_\alpha$ on $U_\alpha$; and its inverse, $\mathcal{O}(D)$ is the subsheaf of the sheaf generated by the inverse rational functions $1/f_\alpha$ on $U_\alpha$. The transition functions of this line bundle which take us from $U_\alpha$ to $U_\beta$ are $f_\alpha/f_\beta$. Denote the Weil divisor $[D]$ and $\text{ord}_V(D)$  the order of the vanishing of an equation for $D$ in the local ring along $V$. Define $[D]$ according to $D$, $[D]=\sum_{\text{cod}(V,X)=1}\text{ord}_V(D)\cdot V$. Of course the concept of order here is rather limited as these rings that are local are valuation rings which are discrete when $X$ is normal. The map $D\mapsto[D]$ denotes as embedding, for normal varieties, of the group of Cartier divisors into the group of Weil divisors. A principal divisor $\text{div}(f)$ for a $f\neq0$ and rational, the local equation is each open set is $f$.
		
		Divisors on the toric variety that map to themselves by $T$, the irreducible subvarieties of codimension one of which are $T$-stable attribute to the rays (or edges) of the fan. If we let $v_i$ be the first lattice point met along edge $\tau_i$, amongst edges $\tau_1,\ldots,\tau_d$, then $D_i=V(\tau_i)$, i.e. these divisors are orbit closures. And for $a_i\in\mathbb{Z}$ the $T$-Weil divisors are the sums $\sum a_iD_i$. 
		
		Let's define $T$-Cartier divisors as, the Cartier divisors that are equivariant by $T$. If we have the affine case where $\text{dim}(\sigma)=n$ and $X=U_\sigma$, then for the fractional ideal $I=\Gamma(X,\mathcal{O}(D))$, $D$ is the divisor that $T$ preserves. And for a unique $u\in\sigma^\smallsmile\cap M$, $I$ is generated by function $\chi^u$. In other words, $I$ is a direct sum of spaces $\mathbb{C}\cdot\chi^u$ over some set of $u$ in $M$, meaning that $I$ is graded by $M$ is implied from the divisor being $T$-invariant. For $\mathfrak{M}=\sum_{u\neq0}\mathbb{C}\cdot\chi^u$, then immediately $\text{dim}I/\mathfrak{M}I=1$ is a must, because $I$ is the principal at determining point $x_\sigma$. Therefore we have that $I=A_\sigma\cdot\chi^u$ for unique $u$. And for some unique $u\in M$, a general $T$-Cartier divisor on $U_\sigma$ is of the form $\text{div}(\chi^u)$.
		
		$\text{ord}_{V(\tau)}(\text{div}(\chi^u))=\langle u,v\rangle$ when $u\in M$ and $v$ is the first lattice point along edge $\tau$. Therefore $[\text{div}(\chi^u)]=\sum_i\langle u,v_i\rangle D_i$. A $T$-Cartier divisor on $U_\sigma$ is of the form $\text{div}(\chi^u)$ for some $u\in M$ when $\text{dim}(\sigma)<n$. However, since $U_\sigma=U_{\sigma'}\times T_{N(\sigma)}$ where $\sigma$ and $\sigma'$ share the same edge and therefore the same Weil divisors, and therefore $T$-Cartier divisors on $U_\sigma$ are in agreement with elements of $M/M(\sigma)$, then $u-u'$ is in $\sigma^\perp\cap M=M(\sigma)$ if and only if $\text{div}(\chi^u)$ and $\text{dim}(\chi^{u'})$ are equal. 
		
		A $T$-Cartier divisor is defined, on a general toric variety $X(\Delta)$, by choosing an element $u(\sigma)\in M/M(\sigma)$ for each $\sigma$ in the fan. And specifying $\text{div}(\chi^{-u})$ (minus sign is taken to conform to the literature) on $U_\sigma$. Another way of saying this is, $\Gamma(U_\sigma,\mathcal{O}(D))=A_\sigma\cdot\chi^{u(\sigma)}$ meaning that $\chi^{u(\sigma)}$ generates the fractional ideal of $\mathcal{O}(D)$ on $U_\sigma$. These are true based on a lack of discrepancy concerning overlaps. Meaning that when $\tau$ is a face of $\sigma$, $u(\sigma)\mapsto u(\tau)$ under the canonical map $M/M(\sigma)\mapsto M/M(\tau)$. Basically, $\sigma_i$ are the maximal cones, and $\text{Ker}(\bigoplus_i M/M(\sigma_j)\rightarrow \oplus_{i<j}M/M(\sigma_i\cap\sigma_j))=\text{lim}_{\leftarrow}M/M(\sigma)=\{\text{$T$-Cartier divisors}\}$.

		
		For any irreducible variety $X$, a homomorphism is given from the group of Cartier divisors on $X$ onto the group of all line bundles $\text{Pic}(X)$, where the group of principal divisors are the kernel. Let's define the group of all Weil divisors modulo $[\text{div}(f)]$, the subgroup of divisors of rational functions, $A_{n-1}(X)$. The homomorphism $\text{Pic}(X)$ to $A_{n-1}(X)$ is determined by the map $D\mapsto[D]$, which when $X$ is normal is an embedding $\text{Pic}(X)\hookrightarrow A_{n-1}(X)$.
		
		A homomorphism is given from $M$ to the group of $T$-Cartier divisors $\text{Div}_TX$ by a principal Cartier divisor determined by any $u\in M$ for a $X$. The computation of $\text{Pic}(X)$ only requires $T$-Cartier divisors and functions. This is similarly true for $A_{n-1}(X)$ with $T$-Weil divisors. Both for complete $X$.
		
		For a fan not contained within a proper subspace of $N_\mathbb{R}$ then we have the exact sequences $0\rightarrow M\rightarrow\text{Div}_TX\rightarrow\text{Pic}(X)\rightarrow0$ and $0\rightarrow M\rightarrow\bigoplus^d_{i=1}\mathbb{Z}\cdot D_i\rightarrow A_{n-1}(X)\rightarrow0$ with embeddings $\text{Div}_TX\hookrightarrow\bigoplus^d_{i=1}\mathbb{Z}\cdot D_i$ and $\text{Pic}(X)\hookrightarrow A_{n-1}(X)$ and the $M$s of each exact sequence are equal. More specifically, for $d$ the number of edges of the fan, $\text{rank}(\text{Pic}(X))\leq\text{rank}(A_{n-1}(X))=d-n$. And in particular, $\text{Pic}(X)$ is free abelian.
		
		Now $\text{Pic}(X(\Delta))\simeq H^2(X(\Delta),\mathbb{Z})$ if all maximal cones of $\Delta$ are $n$-dimensional.
		
		$u(\sigma)\in M/M(\sigma)$ which provides the data for $D$, gives us the continuous piecewise linear function $\Psi_D$ on the support $|\Delta|$. Restricting $\Psi_D$ to $\sigma$ is defined as the linear function $u(\sigma)$: $\Psi_D(V)=\langle u(\sigma),v\rangle$ for $v\in\sigma$. This function is well defined and continuous thanks to the compatibility of the data. 
		Any continuous function given by an element of the lattice $M$, that is on the support $|\Delta|$ which is both integral and linear on each cone, is derived from a unique $T$-Cartier divisor. 
		
		$[D]=\sum-\Psi_D(v_i)D_i$, or otherwise one may say that because $\Psi_D(v_i)=-a_i$, this determines $\Psi_D$ (only when $D=\sum a_iD_i$). $\Psi_{D+E}=\Psi_D+\Psi_E$ and $\Psi_{mD}=m\Psi_D$. The linear function $-u$ is $\Psi_{\text{div}(\chi^u)}$. And $\Psi_D$ and $\Psi_E$ are different only up to the linear function $u\in M$ if $D$ and $E$ are linearly equivalent divisors.
		
		A rational convex polyhedron $P_D$ in $M_\mathbb{R}$ is defined as the $u\in M_\mathbb{R}$ for which $\langle u,v_i\rangle\geq a_i$ for all $i$, or equivalently, the $u\in M_\mathbb{R}$ for which $u\geq\Psi_D$ on the support $|\Delta|$, for $T$-Cartier divisor $D=\sum a_iD_i$ on $X(\Delta)$.
		
		$\Gamma(X,\mathcal{O}(D))=\bigoplus_{u\in P_D\cap M}\mathbb{C}\cdot\chi^u$ represent the global sections of the line bundle $\mathcal{O}(D)$. It follows that in correspondence with the direct sum over the intersection of the $P_D(\sigma)\cap M$, $\bigcap\Gamma(U_\sigma,\mathcal{O}(D))=\Gamma(X,\mathcal{O}(D))$ therefore the above works for restrictions to open sets that are smaller.
		
		When the support is equal to $N_\mathbb{R}$, then the toric variety is complete, in turn on any complete variety, the cohomology groups of a suitable sheaf will be of finite dimension. Specifically, $\Gamma(X,\mathcal{O}(D))$ is of finite dimension and $P_D\cap M$ is also finite when the cones $\Delta$ span $N_\mathbb{R}$.
		
		When is a line bundle generated by its sections? Reconsider the real valued function $\Psi$ on a vector space, it is upper convex if $\Psi(t\cdot v+(1-t)\cdot w)\geq t\psi(v)+(1-t)\psi(w)$ for all $0\leq t\leq1$ and for all vectors $v$ and $w$. The convex function $\Psi$ is strictly convex if for all $n$-dimensional cones, $\text{graph}(\Psi)$ on the complement of $\sigma$ lies strictly beneath $\text{graph}(u(\sigma))$. 
		
		If all maximal cones in $\Delta$ are $n$-dimensional, and we have that $D$ is a $T$-Cartier divisor on the toric variety, then $\Psi_D$ is convex if and only if $\mathcal{O}(D)$ is generated by its sections. We can therefore construct $\Psi_D$ from $P_D$: $$\Psi_D(v)=\text{min}_{u\in P_D\cap M}\langle u,v\rangle = \text{min} \langle u,v \rangle$$ where the vertices of polytope $P_D$ are given by $u_i$.
		
		We have a corresponding mapping $\phi=\phi_D:X(\Delta)\rightarrow\mathbb{P}^{r-1}$ with $x\mapsto(\chi^{u_1}(x):\ldots:\chi^{u_r}(x))$ (where $r=\text{Card}(P_D\cap M)$ and $\mathbb{P}^{r-1}$ is the projective space) for a basis $\chi^u$ for the sections that generate $\mathcal{O}(D)$, where $u\in P_D\cap M$. If $N_\mathbb{R}=|\Delta|$ then $\phi_D$ is an embedding (meaning that $D$ is very ample) if and only if $\Psi_D$ is strictly convex and for every $n$-dimensional cone $\sigma$, the semigroup $S_\sigma$ is generated by $\{u-u_\sigma:u\in P_D\cap M\}$. And in fact on a complete toric variety, $D$ is ample, meaning that any positive multiple of $D$ is very ample. This is true if and only if $\Psi_D$ is strictly convex.
		
		For a polytope that is convex and $n$-dimensional, with vertices in $M$ has a fan $\Delta_P$ with complete toric variety $X_P=X(\Delta_P)$, and corresponding Cartier divisor $D_P$, where $\Psi(v)=\Psi_D(v)=\text{min}_{u\in P}\langle u,v\rangle=\text{min}_{u\in P\cap M}\langle u,v\rangle=\text{min}\langle u_i,v\rangle$, where $u_i$ are the vertices of $P$ and $\chi^{u_i}$ generates $\mathcal{O}(D)$ on $U_{\sigma_i}$, if $\sigma_i$ is the maximal cone for $u_i$.
		
		A fan $\Delta$ is called 'compatible with $P$ - the convex hull of any finite set in $M$', if $\Psi_P(v)=\text{min}_{u\in P}\langle u,v\rangle$ is linear on each $\sigma$ in this complete fan. The line bundle of the determined $D_P$ that is generated by its sections which are linear combinations of $\chi^u$ as $u$ varies over $P\cap M$.
		
		
		For $u\in M_\mathbb{R}$ where $u\geq\Psi$ on $|\Delta|$, $H^0(X,\mathcal{O}(D))_u=\mathbb{C}\cdot\chi^u$ if $u\in P_D\cap M$ and $H^0(X,\mathcal{O}(D))_u=0$ otherwise. And the sections of $\mathcal{O}(D)$ are the graded module: $H^0(X,\mathcal{O}(D))=\bigoplus H^0(X,\mathcal{O}(D))_u$. This can be put more simply by introducing $Z(u)$ to represent the closed conical subset of the support for each $u\in M$ and define it as $Z(u)=\{v\in|\Delta|:\langle u,v\rangle\geq\Psi(v)\}$. Then when $Z(u)$ is exactly equal to the support, then $u$ perfectly belongs to $P_D$. This can be put another way, if we introduce $H^0(|\Delta|\setminus Z(u))$ as the $H^0$ that denotes the $0^\text{th}$ ordinary of sheaf cohomology of the topological space with complex coefficients, then when that vanishes.
		
		As previously defined, $H^0(X,\mathcal{O}(D))=\bigoplus H^0(X,\mathcal{O}(D))_u$ where $H^0(X,\mathcal{O}(D))_u=H^0_{Z(u)}(|\Delta|)$, is the $0^\text{th}$ local cohomology group, in other words the relative group of the pair consisting of the support and the complement of $Z(u)$, and is defined, $H^0_{Z(u)}(|\Delta|)=H^0(|\Delta|,|\Delta|\setminus Z(u))=\text{Ker}(H^0(|\Delta|)\rightarrow H^0(|\Delta|\setminus Z(u)))$. And this is generally true for higher local cohomology groups $H^p_{Z(u)}=H^p(|\Delta|,|\Delta|\setminus Z(u);\mathbb{C})$ and higher sheaf cohomology groups $H^p(X,\mathcal{O}(D))$. And in fact, $H^p(X,\mathcal{O}(D))\simeq\bigoplus H^p(X,\mathcal{O}(D))_u$ are the canonical isomorphisms, where $H^p(X,\mathcal{O}(D))_u=H^p_{Z(u)}(|\Delta|)$ for all $p\geq0$.
		
		If the support $|\Delta|$ and $|\Delta|\setminus Z(u)$ are both convex then all the higher cohomology groups vanish, and this is usually the case for general affine varieties, and where $X$ is affine, the support is a cone and $\Psi$ is linear. This goes further to say that all higher cohomology groups of an ample line bundle on a complete toric variety vanish. And even further, if the support is convex and $\mathcal{O}(D)$ is generated by its sections then $H^p(X,\mathcal{O}(D))$ vanishes for all $p>0$.
		
		If also $X$ is complete, then the Cardinal, $\text{Card}(P_D\cap M)=\text{dim}H^0(X,\mathcal{O}(D))=\sum(-1)^p\text{dim}H^p(X,\mathcal{O}(D))=\chi(X,\mathcal{O}(D))$- with arithmetic genus given $\text{dim}H^0(X,\mathcal{O}_X)=1=\chi(X,\mathcal{O}_X)$.
		
		The $p^\text{th}$ cohomology of, $C^*$, the \u{C}ech complex, is $H^p(\mathcal{O}(D))$, so that $$C^p=\bigoplus_{\sigma_0,\ldots,\sigma_p}H^0(U_{\sigma_0}\cap\ldots\cap U_{\sigma_p},\mathcal{O}(D)),$$ the sum over all cones $\sigma_i$ in the fan, which is equal to $$\bigoplus_{u\in M}\bigoplus_{\sigma_0,\ldots,\sigma_p}H^0_{Z(u)\cap\sigma_0\ldots\sigma_p}(\sigma_0\cap\ldots\sigma_p).$$ Again we are aware that for all cones $\tau$ and for all $u$ for all $i>0$ that $H^i_{Z(u)\cap|\tau|}(|\tau|)=0$.
		
		If $Y$ is the union of a finite number of closed subspaces $Y_j$ and $Z$ is taken to be an arbitrary closed subspace of $Y$, and $Y'=Y_{j_0}\cap\ldots Y_{j_P}$ The sheaf $\mathcal{F}$ is on $Y$ so that $H^i_{Z\cap Y'}=(Y',\mathcal{F})$ vanishes for every $i$ greater than $0$ and all $Y'=Y_{j_0}\cap\ldots Y_{j_p}$. Then for the complex $C^*(\{Y_j\},\mathcal{F})$ whose $p^\text{th}$ term is $C^p(\{Y_j\},\mathcal{F})=\oplus_{j_0,\ldots,j_p}\Gamma_{Z\cap Y_{j_0}\cap\ldots\cap Y_{j_p}}(Y_{j_0}\cap\ldots\cap Y_{j_p}.\mathcal{F})$ where the rows are the resolutions of the complex $\Gamma_Z(Y,\mathcal{I}^*)$ for injective resolution $\mathcal{I}^*$ of $\mathcal{F}$ with double complex $C^*(\{Y_j\},\mathcal{I}^*)$.
		
		And finally for $\Delta'$ a refinement of $\Delta$, we have birational proper map $f:X'=X(\Delta')\rightarrow X=X(\Delta)$ and therefore $f_*(\mathcal{O}_{X'})=\mathcal{O}_X$ and $R^if_*(\mathcal{O}_{X'})=0$ for all $i>0$. So this means that if we take $X'$ to be the resolution of singularities then every $X$ has rational singularities.
		
		Briefly we can say that for $X=U_\sigma$, $\Gamma(X',\mathcal{O}_{X'})=\Gamma(X,\mathcal{O}_X)=A_\sigma$ and since $|\Delta'|=|\sigma|$, then $H^i(X',\mathcal{O}_{X'})=0$ for all $i>0$.

		
		\underline{To create the momentum polytopes from the toric varieties}

		Since toric varieties are naturally defined over integers, we are therefore not restricted to complex toric varieties, as we have been looking at thus far. We can rewrite the algebras with $\mathbb{Z}$ instead of $\mathbb{C}$: $U_\sigma=\text{Spec}(\mathbb{Z}[\sigma^\smallsmile\cap M])$. In fact for any field $K$, and $K$ is the multiplicative semigroup $K^*\cup\{0\}$ (so that for $K=\mathbb{R}\subset\mathbb{C}$, the real points of the toric variety) the semigroup homomorphisms $\text{Hom}_{sg}(\sigma^\smallsmile\cap M,K)$ describe the $K$-valued points of $U_\sigma$.
		
		This is also true for $K$ just a sub-semigroup of the complex field. For $\mathbb{R}_\geq$ then we have the retraction $z\mapsto|z|$. In fact the topological subspace, with retraction $U_\sigma\rightarrow(U\sigma)_\geq$ is determined $(U_\sigma)_\geq=\text{Hom}_{sg}(\sigma^\smallsmile\cap M,\mathbb{R}_\geq)\subset U_\sigma=\text{Hom}_{sg}(\sigma^\smallsmile\cap M,\mathbb{C})$ just as $\mathbb{R}_\geq\subset\mathbb{C}\rightarrow\mathbb{R}_\geq$.
		
		For the varieties we establish a corresponding retraction, $X(\Delta)_\geq\subset X(\Delta)\rightarrow X(\Delta)_\geq$.
		
		An example, for $N$ with partial basis $e_1,\ldots,e_k$ that generate $\sigma$, $(U_\sigma)_\geq$ is isomorphic to $n-k$ copies of $\mathbb{R}$ and $k$ copies of $\mathbb{R}_\geq$. So considering the nonsingular case, $X_\geq$ is therefore a manifold with corners. And when $X$ is singular, looking at the retraction $X_\geq$ can expose worse singularities.
		
		The retraction $X(\Delta)\rightarrow X(\Delta)_\geq$, the result of which is simply the quotient space of $X(\Delta)$ that results from the action of the compact torus $S_N$.
		
		Let's look at when the toric variety is the $n$-dimensional projective space, then $\mathbb{P}^n_\geq=\mathbb{R}_\geq^{n+1}\setminus\{0\}/\mathbb{R}^+$. If we consider the usual covering by affine open sets $U_i=U_{\sigma_i}$ then $(U_i)_\geq$ is made up of the points $(t_0:\ldots:1:\ldots:t_n)$ where $t\geq0$. And therefore $\mathbb{P}^n_\geq=\{(t_0,\ldots,t_n)\in\mathbb{R}^{n+1}\ \text{such that}\ t_i\geq0\ \text{and}\ t_0+\ldots t_n=1\}$. This is the standard simplex and so $$(x_0:\ldots:x_n)\mapsto\frac{1}{\sum|x_i|}(|x_0|,\ldots,|x_n|)$$ defines the retraction $\mathbb{P}^n\rightarrow\mathbb{P}^n_\geq$. The compact torus whose dimension is one less than $\text{Card}(i:t_i\neq0)$ is exactly the fiber over the point $(t_0,\ldots,t_n)$.
		
		The unit circle in $\mathbb{C}^*$ is denoted $S^1$ which is equal to $U(1)$. The compact torus $S_N$ which is $\text{Hom}(M,S^1)$ is clearly a subset of $\text{Hom}(M,\mathbb{C}^*)$ which is the algebraic torus. And so $T_N=S_N\times\text{Hom}(M,\mathbb{R}^+)=S_N\times\text{Hom}(M,\mathbb{R})=S_N\times N_\mathbb{R}$.
		
		Therefore The retraction $(O_\tau)_\geq=X(\Delta)_\geq\cap O_\tau=\text{Hom}(\tau^\perp\cap M,\mathbb{R}^+)$, and as $mathbb{R}^+$ is isomorphic with $\mathbb{R}$, this equals $\text{Hom}(\tau^\perp\cap M,\mathbb{R})=N(\tau)_\mathbb{R}$. Therefore $(O_\tau)_\geq=N(\tau)_\mathbb{R}$ is the quotient space according to the faithful action of $S_{N(\tau)}$ on $O_\tau=T_{N(\tau)}$. It is according to the projection $S_N$ to $S_{N(\tau)}$ that its action on $O_\tau$ is defined. This $S_{N(\tau)}$ (with $\text{dim}S_{N(\tau)}=n-\text{dim}(\tau))$ is exactly the fiber of the retraction $X\rightarrow X_\geq$ over $(O_\tau)_\geq$. Just as how the orbits $O_\tau$ fit together in $X$, so do the retracted spaces $(O_\tau)_\geq$ fit in the retracted variety $X_\geq$. 
		
		Consider the construction. A manifold whose corners are $X_\geq$ creates, if $X$ is complete, a sort of dual polyhedron to the original fan, where each cone of $n$-dimension has a vertex for each $X_\geq$. If cones share a $(n-1)$ face then two vertices join at an edge. And $X_\geq$ is homeomorphic to $P$ if $\Delta=\Delta_P$ comes from a convex polytope in $M_\mathbb{R}$.
		
		Another construction. For $r\in\mathbb{R}^+$ then the mapping $t\mapsto t^r$ is an automorphism of $\mathbb{R}_\geq$. Therefore determining the automorphism $(U_\sigma)_\geq=\text{Hom}_{sg}(\sigma^\smallsmile\cap M,\mathbb{R}_\geq)$ that fit together  to define the homeomorphism of $X_\geq$ to itself. For $r\in\mathbb{Z}^+$, with map $z\mapsto z^r$, it is an endomorphism of $\mathbb{C}$ with similar auto-endomorphism $X_\geq\subset X\rightarrow X_\geq$.
		
		Another construction. In accordance with the action of $T_N$ on $X$, we can consider the action of the quotient $T_N/S_N=N_\mathbb{R}$ on $X/S_N=X_\geq$. And the inclusion $X_\geq\subset X$ that is equivariant since $N_\mathbb{R}=\text{Hom}(M,\mathbb{R}^+)\subset T_N=\text{Hom}(M,\mathbb{C}^*)$.
		
		Another construction. X is the quotient space that arises from $S_N\times X_\geq\rightarrow X$.
		
		Another construction. An example of a deformation retract, for any $\tau$ would be the inclusion $(O_\tau)_\geq\subset(U\tau)_\geq$.

		
		A momentum map can be described as the action of a Lie group on a varieties. It is possible to construct maps for toric varieties and then draw the connection to momentum maps.
		
		A convex polytope $P$ in $M_\mathbb{R}$ with vertices in $M$, gives a toric variety $X(\Delta_P)$ and morphism $\varphi$ from $X$ to the projective space of dimension $r-1$, ($\varphi:X\rightarrow\mathbb{P}^{r-1}$). 
		
		For Cartier divisor $D$ preserved by $T$, $\chi(u)$ is a function for unique $\sigma^\smallsmile\cap M$ that generates the corresponding fractional ideal. 
		
		The momentum map $\mu:X\rightarrow M_\mathbb{R}$ is defined by $$\mu(x)=\frac{1}{\sum|\chi^u(x)|}\sum\limits_{u\in P\cap M}|\chi^u(x)|u$$ where $\chi^u$ for $u\in P\cap M$ are the sections of $\varphi$ and specifically, $|\chi^u(x)|=|\chi^u(t\cdot x)|=|\chi^u(t)|\cdot|\chi^u(x)|=|\chi^u(t\cdot x)|$, for $t\in S_N$ and $x\in X$ and $\mu$ $S_N$-invariant. 
		
		In other words, $$\tilde{\mu}:X_\geq\rightarrow M_\mathbb{R},$$ is the map prompted on the quotient space $X/S_N=X_\geq$ by $\mu$.
		
		Now for $Q$ a face of polytope $P$, with cone of the fan $\sigma$, the above defined $\tilde{\mu}$ does in fact map $(O_\sigma)_\geq\xrightarrow{\simeq}\text{Int}(Q)$ bijectively, where $\text{Int}(Q)$ is the interior of $Q$. If we introduce $\sum|\chi^{u'}(x)|$ as the sum over all $u'$ in $P\cap M$ or in a subset that contains vertices of $P$, then let $\rho_u(x)=|\chi^u(x)|/\sum|\chi^{u'}(x)|$ therefore $\mu(x)=\sum\rho_u(x)u$ where of course $0\leq\rho_u(x)\leq 1$. For $x$ in $(O_\sigma)_\geq$ is also in the inclusion created by extending otside of $\sigma^\perp$ by zero. $\text{Hom}(\sigma^\perp\cap M,\mathbb{R}^+\subset\text{Hom}_{sg}(\sigma^\perp\cap M),\mathbb{R}_\geq)$. For that same $x$, if $u\in Q$ then $\rho_u(x)>0$ and for $u\notin Q$ then $\rho_u(x)$ is exactly zero.
		
		For the finite set of vectors $u_1,\ldots,u_r$ in the dual space $V^*$ of the finite dimensional real vector space $V$, if $K$ is the convex hull of these vectors, and is not contained in a hyperplane. Then for $\varepsilon_1,\ldots,\varepsilon_r$ that represent any positive numbers, then the map $\rho_i$ that takes $V$ to $\mathbb{R}$ is defined as $$\rho_(x)=\varepsilon_ie^{u_i(x)}/(\varepsilon_1e^{u_1(x)}+\ldots+\varepsilon_re^{u_r(x)})$$ and therefore the map $\mu:V\rightarrow V^*$ is the real analytic isomorphism of $V$ onto $\text{Int}K$ and is defined, $\mu(x)=\rho_1(x)u_1+\ldots+\rho_r(x)u_r$.
		
		In fact due to the actions of $T_N$ on $X_P$ and $(\mathbb{C}^*)^r$ on $\mathbb{P}^{r-1}$ then we have a map from $X_P$ to $\mathbb{P}^{r-1}$, where $T_N=\text{Hom}(M,\mathbb{C}^*)\rightarrow\text{Hom}(\mathbb{Z}^r,\mathbb{C}^*)=T$ by the map $\mathbb{Z}^r\rightarrow M$ which takes basic vectors to points of $P\cap M$. And therefore we have a momentum map $\mathfrak{M}:\mathbb{P}^{r-1}\rightarrow\text{Lie}S^*=\mathbb{R}^*=\mathbb{R}$ for the action of $(S^1)^r$ on $\mathbb{P}^{r-1}$. Specifically we have $$\mathfrak{M}(x)=\frac{1}{\sum|x_i|^2}\sum^r_{i=1}|x_i|^2e_i^* \label{mmapfultonpoint}$$ up to a scalar factor, for $x\in\mathbb{P}^{r-1}$ and $v=(x_1,\ldots,x_r)\in \mathbb{C}^r$.
		
		And in fact, $$X_P\xrightarrow{\phi}\mathbb{P}^{r-1}\xrightarrow{\mathfrak{P}}(\mathbb{R}^r)^*\rightarrow M_\mathbb{R}$$.
		
	\end{proof}
	
	\section*{Acknowledgements}
	The author is funded by the Marie-Curie individual fellowship and would like to thank the European Commission for fully funding project 898145 — ROBOTTOPES. The author would also like to thank Professor Manuel de Leon for bringing working on contact manifolds into consideration and reading drafts. The author would also like to thank Yael Karshon for some consultations.
	
	
	
	

	
\end{document}